\documentclass[11pt,reqno]{amsart}
\usepackage[pagewise]{lineno}
\usepackage{amsmath}
\usepackage{amsfonts}
\usepackage{amssymb}
\usepackage[dvips]{graphicx}
\usepackage{color}
\usepackage[pagewise]{lineno}
\usepackage{url}
\theoremstyle{plain}
\newtheorem{athm}{Theorem}

\newtheorem{sublemma}{Sublemma}[section]
\newtheorem{theorem}{Theorem}[section]
\newtheorem{proposition}[theorem]{Proposition}
\newtheorem{corollary}[theorem]{Corollary}
\newtheorem{lemma}[theorem]{Lemma}

\theoremstyle{definition}

\newtheorem{example}{Example}[section]
\newtheorem{definition}[theorem]{Definition}

\newtheorem{remark}[theorem]{Remark}

\DeclareMathOperator{\esssup}{ess \ sup}

\DeclareMathOperator{\e}{e}
\DeclareMathOperator{\id}{id}
\DeclareMathOperator{\inte}{int}
\input{tcilatex}

\begin{document}

\begin{abstract}
We consider a class of endomorphisms that contains a set of piecewise partially hyperbolic dynamics semi-conjugated to non-uniformly expanding maps. Our goal is to study a class of endomorphisms that preserve a foliation that is almost everywhere uniformly contracted, with possible discontinuity sets parallel to the contracting direction. We apply the spectral gap property and the $\zeta$-H\"older regularity of the disintegration of its equilibrium states to prove a quantitative statistical stability statement. More precisely, under deterministic perturbations of the system of size $\delta$, we show that the $F$-invariant measure varies continuously with respect to a suitable anisotropic norm. Furthermore, we establish that certain interesting classes of perturbations exhibit a modulus of continuity estimated by $D_2\delta^\zeta \log \delta$, where $D_2$ is a constant.

\end{abstract}

\title[Quantitative Stability for Equilibrium States of Skew Products]{Quantitative Statistical Stability for Equilibrium States of Piecewise Partially Hyperbolic Maps.}
\author[Rafael A. Bilbao]{Rafael A. Bilbao}
\author[Ricardo Bioni]{Ricardo Bioni}
\author[Rafael Lucena]{Rafael Lucena}

\date{\today }

\address[Rafael A. Bilbao]{Universidad Pedag\'ogica y Tecnol\'ogica de Colombia, Avenida Central del Norte 39-115, Sede Central Tunja, Boyac\'a, 150003, Colombia.}
\email{rafael.alvarez@uptc.edu.co}

\address[Ricardo Bioni]{Rua Costa Bastos, 34, Santa Teresa, Rio de Janeiro-Brasil}
\email{ricardo.bioni@hotmail.com}

\address[Rafael Lucena]{Universidade Federal de Alagoas, Instituto de Matemática - UFAL, Av. Lourival Melo Mota, S/N
	Tabuleiro dos Martins, Maceio - AL, 57072-900, Brasil}
\email{rafael.lucena@im.ufal.br}
\urladdr{www.im.ufal.br/professor/rafaellucena}
\maketitle

\section{Introduction}\label{intro}

Understanding how statistical properties change when a system is perturbed is of significant interest in both pure and applied mathematics. When a statistical property of a system varies continuously after deterministic or even stochastic variations, we say it is \textit{statistically stable}. The study of these properties is motivated by the desire to understand how uncertainty impacts the quantitative and qualitative measurements of systems.

%An important ergodic object of a dynamical system that makes interesting the investigation of its
%stability is the invariant measure, given that it is key in understanding the long-term behavior of the dynamics.
%The invariant measure of a dynamic system is a crucial ergodic object for investigating its statistical stability, as it establishes numerous long-term system characteristics. 
An important ergodic object of a dynamical system that makes interesting the investigation of its
stability is the invariant measure, given that it is key in understanding the long-term behavior of the dynamics. To do this, consider a one-parameter family of dynamics $\{F_{\delta }\}_{\delta \in [0,1)}$ as a perturbation of a system $F=F_0$. Suppose that $\{ F{_\delta }\}_{\delta \in [0,1)}$ admits a one-parameter family of invariant measures $\{\mu_\delta \}_{\delta \in [0,1)}$, i.e., $\mu_\delta$ is a $F_{\delta}$-invariant probability measure for all $\delta \in [0,1)$. We say that $\mu _0$ is \textit{statistically stable} if the function $\delta \mapsto \mu _\delta$ is continuous at $0$ in a suitable topology. In this paper, our aim is to prove the continuity and to estimate the modulus of continuity of the function $\delta \mapsto \mu _\delta$ at $0$.

To prove our results, we aim to construct a suitable vector space of signed measures that satisfy three main properties. This space includes the family $\{\mu_\delta \}_{\delta \in [0,1)}$, where $\mu_\delta$ is the unique $F_\delta$-invariant measure inside this space for all $\delta \in [0,1)$. Additionally, the function $\delta \longmapsto \mu _\delta$ is continuous at $0$, with a modulus of continuity of the order of $\delta^\zeta \log \delta$ (where $\zeta$ is a constant that depends on $F$). We use the functional analytic approach of \cite{GLU} to study the transfer operator of $\operatorname{F}_*$, which is the linear operator $\operatorname{F}_*$ that associates each signed measure $\mu$ with the signed measure $\operatorname{F}_* \mu$ defined by $\operatorname{F}_* \mu (A) = \mu (F^{-1}(A))$. Positive fixed points of $\operatorname{F}_*$ are $F$-invariant measures, so the problem is reduced to understanding how the eigenvectors of the induced family of transfer operators $\{\operatorname {F}_{\delta*}\}_{\delta \in [0,1)}$ associated with unitary eigenvalues (invariant measures) vary when the system changes.

In \cite{GLU}, the authors studied this property for Lorenz-like systems, $F=(f,G)$. In this case, the quotient the quotient map $f$ is a piecewise expanding map and the fiber function $G$ is Lipschitz in the first variable on each element of a finite family of vertical strips. This family covers a full measure set of the ambient space. Moreover, they defined anisotropic spaces where the action of the transfer operator had a spectral gap, which provided a quantitative stability statement, estimating the modulus of continuity of the function $\delta \longmapsto \mu _\delta$ at $0$. It is worth mentioning that the Bounded Variation regularity of the disintegration of the invariant measure was a crucial ingredient used in \cite{GLU} to obtain the stability result.

The dynamical system under consideration in this work is a skew-product of the form $F=(f,G)$, where the quotient map $f$ is a non-uniformly expanding system. Moreover, the fiber function $G$ is H\"older in the first variable on each element of a finite family of vertical strips. As in \cite{GLU}, this family covers a full measure set of the ambient space. To handle this system, we utilize the H\"older regularity of the disintegration of the invariant measure established in \cite{RRR}. In addition, to overcome the challenges posed by the new hypotheses, we introduce several new definitions, including the concepts of an \textbf{admissible $R(\delta)$-perturbation} and a \textbf{$(R(\delta), \zeta)$-family of operators} (see the next paragraphs for these definitions). Some of these definitions generalize the ones given in \cite{GLU}.

Although the uniform hyperbolic scenario is well understood, our understanding of partially hyperbolic systems, especially those that are non-invertible or have discontinuities, is far from complete. For further information on this topic, interested readers can refer to \cite{Gjep} and \cite{TC}. While the former deals with systems that allow discontinuities, the latter is restricted to smooth invertible systems.

In the approach presented here, a finite number of sets of discontinuities (lines) parallel to the contracting direction are allowed. In comparison with the works cited above, \cite{GLU} requires a uniform expansion (piecewise) on the base map $f$ despite allowing discontinuities. On the other hand, \cite{Gjep} obtains quantitative estimates for statistical stability for piecewise constant toral extensions $F = (f, G)$ with a uniform expanding quotient map $f$ and uses norms similar to those employed in our work.

In this paper, we study skew-product maps $F: \Sigma \longrightarrow \Sigma$, where $F(x,y)=(f(x),G(x,y))$, $\Sigma = M \times K$ is a product space, and $M$ is a compact and connected Riemannian manifold equipped with a Riemannian metric $d_1$, while $K$ is a compact metric space equipped with a metric $d_2$. The space $\Sigma$ is endowed with the metric $d_1 + d_2$. For simplicity, we assume that $\mathrm{diam}(M) = 1$, which is not restrictive but will avoid multiplicative constants. 

We also assume that $F$ contracts almost every vertical fiber $\gamma = \{x\}\times K$ and its quotient map $f: M\longrightarrow M$ is a non-uniformly expanding system. More precisely, $f:M \longrightarrow M$ is a local diffeomorphism, and there exists a continuous function $L:M\longrightarrow\mathbb{R}$ such that for every $x\in M$, there exists a neighborhood $U_x$ of $x$ such that $f_x:=f|_{U_x}: U_x \longrightarrow f(U_x)$ is invertible and satisfies
$$
d_1(f_x^{-1}(y), f_x^{-1}(z)) \leq L(x)d_1(y, z)
$$
for all $y,z \in f(U_x)$. In particular, $\#f^{-1}(x)$ is constant for all $x\in M$, and we set $\deg(f):=\#f^{-1}(x)$ as the degree of $f$.

Define the function $\rho: M \longrightarrow \mathbb{R}$ by 
\begin{equation*}
	\rho(x) := \dfrac{1}{|\det Df(x)|},
\end{equation*}
where $\det Df$ is the Jacobian of $f$ with respect to a fixed probability $m_1$ on $M$ (see \cite{Kva} for definitions and basic results on the Jacobian). We assume that $m_1$ is an equilibrium state for the potential $\phi = \log \dfrac{1}{|\det D f|}$. That is, $m_1$ satisfies 
\begin{equation}\label{eq1}
	\int{\log \dfrac{1}{|\det Df|}}dm_1 + h_{m_1}(f) = \sup _{f \ast \mu = \mu} \left\{\int {\log \dfrac{1}{|\det D f|}} d\mu + h_{\mu}(f)\right\},
\end{equation}
where $h_{\mu}(f)$ denotes the entropy of the system $(f, \mu)$. By \cite{VAC} and \cite{VM} a measure $m_1$ which satisfies (\ref{eq1}) exists. Note that $\rho$ is defined $m_1$-a.e. $x \in M$.

Suppose that there exists an open region $\mathcal{A} \subset M$ and constants $\sigma > 1$ and $L_1 \geq 1$ such that the following conditions hold:
\begin{enumerate}
	\item[(f1)] $L(x) \leq L_1$ for every $x \in \mathcal{A}$ and $L(x) < \sigma^{-1}$ for every $x \in \mathcal{A}^c$. Moreover, the constant $L_1$ satisfies the inequality given by equation (\ref{kdljfhkdjfkasd}) that will be presented ahead.
	\item[(f2)] There exists a finite covering $\mathcal{U}$ of $M$ by open domains of injectivity of $f$, such that $\mathcal{A}$ can be covered by $q < \deg(f)$ of these domains.
\end{enumerate}

Let $H_\zeta$ ($\zeta \leq 1$) represent the set of $\zeta$-Hölder functions $h: M \longrightarrow \mathbb{R}$. In other words, defining $$H_\zeta(h) := \sup_{x\neq y} \frac{|h(x) - h(y)|}{d_1(x,y)^\zeta},$$ we have $$H_\zeta := \{ h:M \longrightarrow \mathbb{R}: H_\zeta(h) < \infty\}.$$

Next, we require that condition (f3) holds, which is an open condition with respect to the Hölder norm. Equation (\ref{f32}), presented below, specifies that $\rho$ is $\zeta$-Holder and belongs to a small cone of Hölder continuous functions (see \cite{VAC}). For examples of non-uniformly expanding transformations that satisfy (f1), (f2) and (f3), the reader may refer to \cite{VAC} and \cite{VM}. In this paper, we also explore such maps in Section \ref{ooisidrosr} and provide Example \ref{new}, where (f1), (f2), and (f3) have been explicitly demonstrated. 

\begin{enumerate}	
	\item[(f3)] There exists a small enough $\epsilon _\rho >0$ such that
	\begin{equation}\label{f31}
		\sup \log (\rho) - \inf \log (\rho) <\epsilon _\rho;
	\end{equation}and
	\begin{equation}\label{f32}
		H_\zeta (\rho) < \epsilon_\rho \inf \rho.
	\end{equation}
\end{enumerate}

Precisely, we assume that the constants $\epsilon_\rho$ and $L_1$ satisfy the condition:
\begin{equation}\label{kdljfhkdjfkasd}
	\e ^{\epsilon_\rho} \cdot \left( \dfrac{(\deg(f) - q)\sigma ^{-\zeta} + qL_1^\zeta[1+(L_1-1)^\zeta] }{\deg(f)}\right)< 1.
\end{equation}

We assume that the fiber map $G: \Sigma \longrightarrow K$ satisfies:

\begin{enumerate}
	\item[(G1)] $G$ is uniformly contracting on $m_1$-a.e. vertical fiber $\gamma_x :={x}\times K$. Precisely, there exists $0< \alpha <1$ such that for $m_1$-a.e. $x\in M$, it holds that
	\begin{equation}
		d_2(G(x,z_{1}),G(x,z_2))\leq \alpha d_2(z_{1},z_{2}), \quad \forall z_{1},z_{2}\in K. \label{contracting1}
	\end{equation}
	We denote the set of all vertical fibers $\gamma_x$ by $\mathcal{F}^s$: $$\mathcal{F}^s:= \{\gamma _x:=\{ x\}\times K; x \in M \} .$$ When no confusion is present, the elements of $\mathcal{F}^s$ will be denoted simply by $\gamma$ instead of $\gamma _x$.
	\item[(G2)] Let $P_1, \cdots, P_{\deg(f)}$ be the partition of $M$ given in Remark \ref{chkjg}, and let $\zeta \leq 1$. Suppose that  
	\begin{equation*}
		|G_i|_\zeta:= \sup _y\sup_{x_1, x_2 \in P_i} \dfrac{d_2(G(x_1,y), G(x_2,y))}{d_1(x_1,x_2)^\zeta}< \infty.
	\end{equation*}
	Denote by $|G|_\zeta$ the following constant: 
	\begin{equation}\label{jdhfjdh}
		|G|_\zeta := \max_{i=1, \cdots, \deg(f)} \{|G_i|_\zeta\}.
	\end{equation}
\end{enumerate}

\begin{remark}
	The condition (G2) implies that $G$ may be discontinuous on the sets $\partial P_i \times K$ for all $i=1, \cdots, \deg(f)$, where $\partial P_i$ denotes the boundary of $P_i$.
\end{remark}

\begin{remark}\label{notation}
	Since there is a bijective correspondence between the elements $x \in M$ and the fibers $\gamma = \{x\}\times K$, from now on we also use $\gamma$ to denote the elements of $M$.
\end{remark}

For the system $F$ under consideration (see \cite{RRR}), the transfer operator $\operatorname {F}_\ast$ has a spectral gap on a space of signed measures, $\mu$, such that its projection ($\pi_ {1}(x,y)=x$ for all $x \in M$ and $y \in K$) onto the first coordinate, $\pi_ {1\ast} \mu$ (the pushforward), is absolutely continuous with respect to $m_1$ and its density satisfies $\frac{d \pi _{1\ast} \mu}{d m_1} \in H_\zeta$. We denote this space by $S^\infty$ (see definition \ref{sdfsdfsdasd}). It is worth mentioning that $S^\infty$ is a suitable anisotropic space of disintegrated measures. For these maps, we prove quantitative results on the statistical stability of the unique equilibrium state of $F$ in $S^\infty$ under a class of deterministic perturbations of the system, $\{F_{\delta }\}_{\delta \in [0,1)}, F_{\delta }=(f_\delta, G_\delta)$, $f_\delta:M \longrightarrow M$, $G_\delta : M \times K \longrightarrow K$ for all $\delta \in [0,1)$ and $F_0=F$.

For such a perturbation, we suppose the following conditions:

\begin{enumerate}
	\item [(U1)] There exists a small enough $\delta _1 $ such that for all $\delta \in (0, \delta_1)$,  it holds
	$$\displaystyle{ \deg (f_\delta)=\deg(f)},$$ for all $\delta \in (0, \delta_1).$
	
\end{enumerate}

\begin{enumerate}
	
	\item [(U2)] For every $\gamma \in M$ and for all $i=1,\cdots,\deg(f)$ denote by $\gamma_{\delta,i}$ the $i$-th pre-image of $\gamma$ by $f_\delta$ (see Remark \ref{notation}). Suppose there exists a real-valued function $\delta \longmapsto R(\delta) \in \mathbb{R}^+$ such that $$\lim_{\delta \rightarrow 0^+} {R(\delta)\log (\delta)}=0$$ and the following three conditions hold:

	\begin{enumerate}
		\item [(U2.1)]

		$\displaystyle{\sum_{i=1}^{\deg(f)}%
			\left\vert \dfrac{1}{\det Df_{\delta}(\gamma _{\delta,i})}
			-\dfrac{1}{\det Df_{0}(\gamma _{0,i})}\right\vert \leq R(\delta)}$;

	\end{enumerate}	
	
	\begin{enumerate}
		\item [(U2.2)] 
		
		$\esssup _\gamma \max_{i=1, \cdots, \deg(f)}	d_1(\gamma _{0,i},\gamma _{\delta,i}) \leq R(\delta);$
	\end{enumerate}

	\begin{enumerate}
		\item [(U2.3)] $G_0$ and $G_\delta$ are $R(\delta)$-close in the $\sup$ norm: for all $\delta$  $$d_2(G_{0}(x,y),G_{\delta }(x,y))\leq R(\delta) \ \forall (x,y) \in M\times K;$$ 
	\end{enumerate}
	
	\item [(U3)] For all $\delta \in (0,\delta_1)$, $f_\delta$ has an equilibrium state $m_{1,\delta}$, and $m_{1,\delta}$ is equivalent to $m_1$ for all $\delta \in [0,\delta_1)$. This implies that $m_1 \ll m_{1,\delta}$ and $m_{1,\delta} \ll m_1$ for all $\delta \in [0,\delta_1)$.  
\end{enumerate}

\begin{remark}\label{toyiout}
	By (U3), note that (see Remark \ref{notation}) $\sum_{i=1}^{\deg(f)}  \dfrac{1}{\det Df_{\delta}(\gamma _{\delta,i})} =1$ $m_1$-a.e., since $m_1 \ll m_{1,\delta}$ for all $\delta$ and $m_{1,\delta}$ is $f_\delta$-invariant. 
\end{remark}
\begin{enumerate}
	\item [(A1)] There exist constants $D>0$ and $0<\lambda<1$ such that for all $g \in H_\zeta$,  all $\delta \in [0,1)$,  and all $n \geq 1$, the following inequality holds:
	
	\begin{equation*}
		|\operatorname{P}_{f_\delta}^{n}g|_{\zeta} \leq D \lambda  ^n | g|_{\zeta} + D|g|_{\infty},
	\end{equation*}where $|g|_\zeta := H_\zeta (g) + |g|_{\infty}$ and $\operatorname{P}_{f_{\delta }}$ is the Perron-Frobenius operator of $f_{\delta }$. That is, for all $\delta \in [0,1)$,  $\operatorname{P}_{f_{\delta }}$ is the unique linear operator $\operatorname{P}_{f_{\delta }}: L^1_{m_1} \longrightarrow L^1_{m_1}$ such that for all $\psi \in L^{\infty} _{m_1}$ and all $\phi \in L^1 _{m_1}$ it holds that $$\int{\psi \operatorname{P}_{f_{\delta }} (\phi)}dm_1 = \int{(\psi \circ f_\delta) \phi}dm_1.$$ 
\end{enumerate}

\begin{enumerate}
	\item [(A2)] For all $\delta \in [0,1), $ let $\alpha _\delta$,  $L_{1,\delta}$ and $|G_\delta|_ \zeta$ be the contraction rate $\alpha$ given by Equation (\ref{contracting1}) for $G_\delta$,  the constant $L_1$ given by (f1) for $f_\delta$, and the constant $|G|_ \zeta$  defined by Equation $(\ref{jdhfjdh})$, respectively.  Set  $ \beta_\delta:= (\alpha_\delta L_{1,\delta})^\zeta$ and $D_{2, \delta}:=\{\epsilon _{\rho, \delta} L_{1,\delta}^\zeta + |G_\delta|_ \zeta L_{1,\delta}^\zeta\}$. Suppose that, $$\sup _ \delta \beta_\delta <1$$and $$\sup _ \delta D_{2, \delta} < \infty.$$ 
\end{enumerate}We define an \textbf{admissible $R(\delta)$-perturbation} as a family $\{F_{\delta }\}_{\delta \in [0,1)}$, where $F_{\delta }$ satisfies conditions (U1), (U2), (U3), (A1), (A2), as well as (f1), (f2), (f3), (G1), and (G2) for all $\delta$.

\noindent\textbf{Statements of the Main Results.} In this section, we present the main results of this article and provide an explanation of how Theorems \ref{belongss}, \ref{htyttigu}, and Corollary \ref{htyttigui} are proven.

The first result guarantees the existence and uniqueness of an invariant measure for $F$ in the space $S^{\infty}$, which is an equilibrium state if $F$ is continuous.  In particular, all admissible $R(\delta)$-perturbation, $\{F_{\delta }\}_{\delta \in [0,1)}$, has a family of $F_{\delta }$-invariant measures, $\{\mu_{\delta }\}_{\delta \in [0,1)}$.

\begin{athm}
	\label{belongss} The system $F$ has a unique invariant probability, $\mu_0 \in S^{\infty }$. If $F$ is continuous, then $\mu_0$ is an equilibrium state. 
\end{athm}

The next Theorem \ref{htyttigu} gives a relation between an admissible $R(\delta)$-perturbation, $\{F_{\delta }\}_{\delta \in [0,1)}$, and the variation of the induced family of invariant measures, $\{\mu_{\delta }\}_{\delta \in [0,1)}$. Moreover, it estimates the modulus of continuity on $0$ of the induced function $\delta \longmapsto  \mu_{\delta }$, given by 
\begin{equation*}
	\delta \longmapsto F_\delta \longmapsto \mu _\delta, \ \ \delta \in [0,1),
\end{equation*}with respect to the norm $||\cdot||_\infty$ defined by $$||\mu||_\infty:=\displaystyle {\sup _{\gamma, g}}\{ \left| \int {g} d\mu|_\gamma \right| \},$$ where $\gamma \in M$, $g$ ranges over the set $H_\zeta$ satisfying $H_\zeta(g)\leq 1$, $|g|_\infty \leq 1$ and $\mu|_\gamma$ is defined from the conditional measure $\mu _\gamma$ of the disintegration of $\mu$ along $\mathcal{F}^s$  (see Definition \ref{sdfsdfsdasd}). 

\begin{athm}[Quantitative stability for deterministic perturbations]
	Let $\{F_{\delta }\}_{\delta \in [0,1)}$ be an admissible $R(\delta)$-perturbation.  Denote by $\mu_\delta$ the invariant measure of $F_\delta$ in $S^\infty$, for all $\delta$. Then, there exist constants $D_2 < 0$ and $\delta _1 \in (0,\delta_0)$ such that for all $\delta \in [0,\delta _1)$, it holds
	\begin{equation}\label{stabll}
		||\mu_{\delta }-\mu_{0}||_{\infty}\leq D_2R(\delta)^\zeta \log \delta .
	\end{equation}
	\label{htyttigu}
\end{athm}

Many interesting perturbations of $F$ ensure the existence of a linear $R(\delta)$. For instance, perturbations with respect to topologies defined in the set of the skew-products, induced by the $C^r$ topologies. Therefore, if the function $R(\delta)$ is of the type, $$R(\delta)=K_6 \delta,$$ for all $\delta$ and a constant $K_6$ we immediately get the following corollary.  

\begin{corollary}[Quantitative stability for deterministic perturbations with a linear $R(\delta)$]
	Let $\{F_{\delta }\}_{\delta \in [0,1)} $ be an admissible $R(\delta)$-perturbation, where $R(\delta)$ is defined by $R(\delta)=K_5\delta$. Denote by $\mu_\delta$ the unique invariant probability of $F_\delta$ in $S^\infty$, for all $\delta$. Then, there exist constants $D_2 <0$ and $\delta _1 \in (0,\delta_0)$ such that for all $\delta \in [0,\delta _1)$, it holds\footnote{A question to be answered is: is $O(\delta^\zeta \log \delta )$ an optimal modulus of continuity?}
	\begin{equation*}
		||\mu_{\delta }-\mu_{0}||_{\infty}\leq D_2\delta^\zeta \log \delta.
	\end{equation*}
	\label{htyttigui}
\end{corollary}

\noindent\textbf{Plan of the paper.} The paper is structured as follows:

\begin{itemize}
	\item Section \ref{intro}: in this section, we introduce the type of systems we consider in the paper. Essentially, it is a class of systems that comprises a set of piecewise partially hyperbolic dynamics $F(x,y) = (f(x), G(x,y))$. The system $F$ has a non-uniformly expanding basis map $f$, and a fiber map $G$ that uniformly contracts $m_1$-a.e. vertical fiber $\gamma \in M$, where $m_1$ is an $f$-invariant equilibrium state. Additionally, in this section, we state the main results and definitions. For instance, we introduce the concept of \textbf{admissible $R(\delta)$-perturbations};
	
	\item Section \ref{ooisidrosr}: we present some examples;
	
	\item Section \ref{seccc}: we present some tools and preliminary results, some of which have already been published in the literature. Most of them are from \cite{GLU} and \cite{RRR}. We use these tools to introduce the functional spaces discussed in the previous paragraphs; 
	\item Section \ref{kjrthkje}: we prove some basic results satisfied by \textbf{admissible} \linebreak[4] $R(\delta)$-\textbf{perturbations} which are important to obtain Theorem \ref{htyttigu} and Lemma \ref{thshgf}; 
	\item Section \ref{perturbationoperators}: we introduce the definition of \textbf{$(R(\delta), \zeta)$-family of operators} and present results relating this family to \textbf{admissible $R(\delta)$-perturbations};
	\item Section \ref{sofkjsdkgfhksjfgd}: we prove Theorem \ref{belongss};
	\item Section \ref{kkdjfkshfdsdfsttr}: we prove Theorems \ref{htyttigu} and Corollary \ref{htyttigui}.
	
\end{itemize}

\noindent\textbf{Acknowledgments.} We are thankful to Stefano Galatolo,  Wagner Ranter and  Davi Lima for all the valuable comments and fruitful discussions regarding this work.

Special thanks are extended to Sophia Homolka, Phillip Homolka, Krerley Oliveira and Gabriel Montoya.

This work was partially supported by CNPq (Brazil) Grants 409198/2021-8 and Alagoas Research Foundation-FAPEAL (Brazil) Grants 60030.0000000161/2022.

\section{Examples}\label{ooisidrosr}
In what follows, we present some examples which satisfy the assumptions described in the previous section. 

\begin{example} \label{new}

	Let $f_0: M \longrightarrow M$ be a map defined by $f_0(x,y)=(\id(x),3y\mod1)$, where $M:=[0,1]^2$ is endowed with the $\mathbb{T}^2$ topology and $\id$ is the identity map on $[0,1]$. On $[0,1]^2$, we consider the metric $d_1((x_0, y_0), (x_1,y_1)) = \max {d(x_0,y_0), d(x_1,y_1)}$, where $d$ is the metric of $[0,1]$. This system has $(0,0) = (1,1)$ and all points of the horizontal segment $[0,1]\times \{1/2\}$ as fixed points.
	
	Consider the partition $P_0= [0,1/3] \times [0,1]$, $P_1= [1/3,2/3] \times [0,1]$, and $P_2= [2/3,1]\times[0,1]$. In particular, the fixed point $p_0=(1/2,1/2) \in \mathcal{A}:= \inte P_1$ (where $\inte P_1$ means the interior of $P_1$).
	
	For a given $\delta>0$, consider a perturbation $f$ of $f_0$, given by $f(x,y)=(g(x), 3y\mod1)$ such that $g(1/2)=1/2$, $0< g'(1/2) <1$ and $g$ is $\delta$-close to $\id(x)=x$ in the $C^2$ topology. Moreover, suppose that $|g'(x)| \geq k_0 > 1$ for all $x \in P_0 \cup P_2$. In particular, without loss of generality, suppose that $1 - \delta < g'(1/2) < 1+ \delta <3$ and $\deg(g)=1$. Below, the reader can see the graph of such a function $g$.

	\begin{figure}[!htb]
		\centering
		\includegraphics[width=0.5\textwidth]{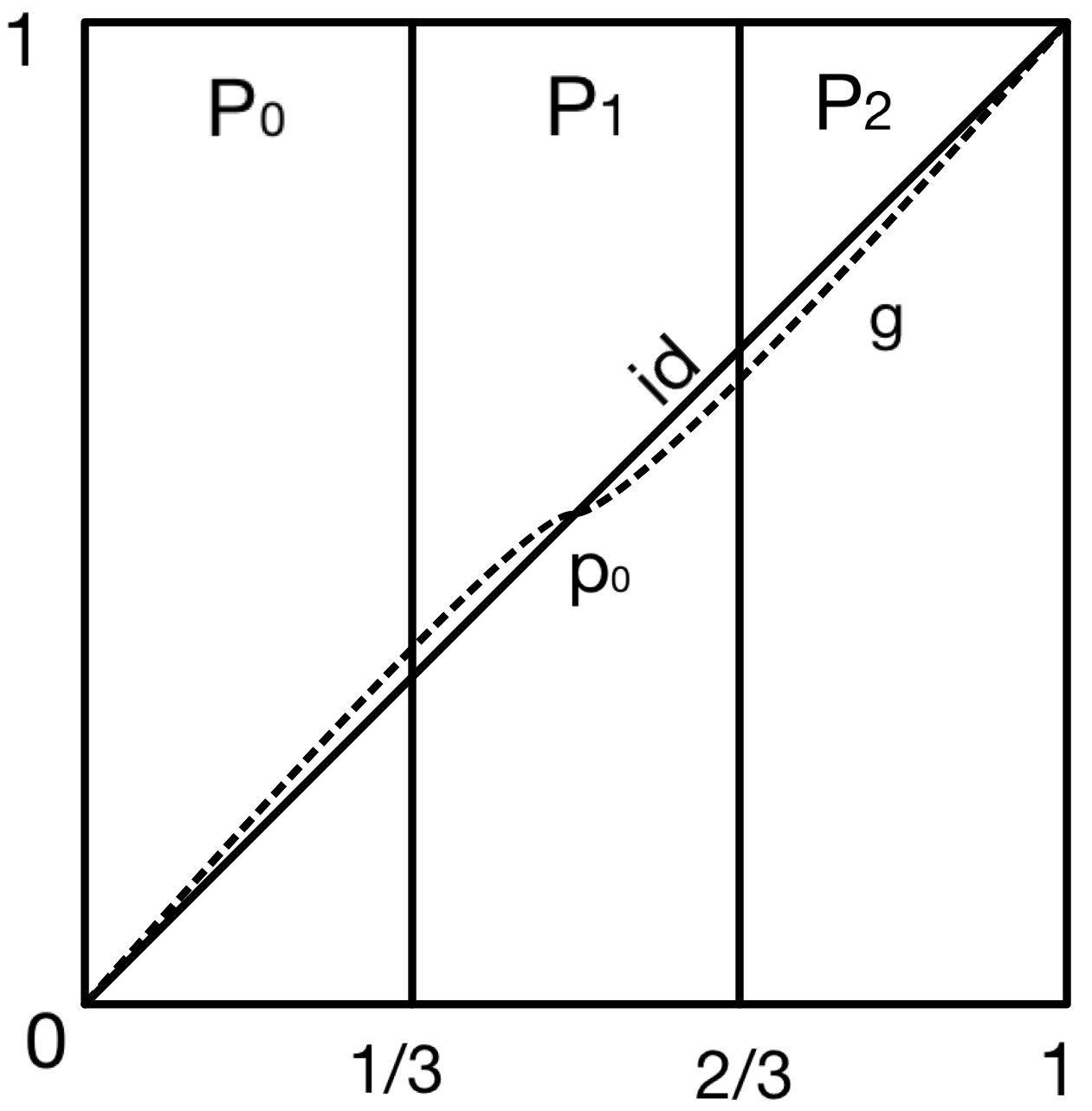}
		\caption{The graph of the perturbed map $g$.}
	\end{figure}Thus, we have
	$$
	Df (1/2, 1/2)=\left( \begin{array}{ll}
		g'(1/2) & 0 \\
		0 & 3
	\end{array}\right).
	$$And since $p_0=(1/2,1/2)$ is still a fixed for $f$ we have that $p_0$ becomes a saddle point (for $f$) as in the next Example \ref{sesprowerpo}. Moreover, since $\deg(g)=1$, we have that $\deg(f_0)=\deg(f)=3$, $q=1$, $\sigma = 3$, $L(x,y) := 1/ g'(x), \forall (x,y) \in \mathbb{T}^2$. In general, we have the following expression for the derivative of $f$:
	$$
	Df(x,y)= \left( \begin{array}{ll}
		g'(x) & 0 \\
		0 & 3
	\end{array}\right),
	$$ for all $(x,y) \in [0,1]^2$. This expression ensures that $\rho$ is $\zeta$-Holder for $0< \zeta \leq 1$. Therefore, for every $\epsilon>0$, there exists $\delta>0$ such that $L(x,y) \in (1-\epsilon, 1+\epsilon)$ for all $(x,y) \in [0,1]^2$, so that $L_1$ can be defined as $L_1 := 1 + \epsilon$.
	
	Since $g:[0,1] \longrightarrow [0,1]$ is $\delta$-close to $\id:[0,1] \longrightarrow [0,1]$, we have that $1-\delta < g'(x) < 1+\delta$ for all $x \in [0,1]$ and $3(1-\delta) < \det Df(x,y) < 3(1+\delta)$ for all $(x,y) \in [0,1]^2$. Thus, 
\begin{eqnarray*}
\sup _{(x,y)} \log \dfrac{1}{\det Df (x,y)} - \inf _{(x,y)} \log \dfrac{1}{\det Df (x,y)} &\leq& \log \dfrac{1}{3(1-\delta)} - \log \dfrac{1}{3(1+\delta)} \\&=& \log \dfrac{(1+\delta)}{(1-\delta)}.
\end{eqnarray*}
Therefore, it holds
\begin{equation}\label{dfdfds}
\sup _{(x,y) \in [0,1]^2} \log \dfrac{1}{\det Df (x,y)} - \inf _{(x,y) \in [0,1]^2} \log \dfrac{1}{\det Df (x,y)} \leq \log \dfrac{(1+\delta)}{(1-\delta)}.
\end{equation}Note that, since $\e ^{\epsilon _\rho} \approx 1$, $L_1 \approx 1$, $0<\zeta \leq 1$ and $q(L_1^\zeta[1+(L_1-1)^\zeta]) \approx 1$ we have that 
\begin{equation}\label{jhdjhjagsdfh}
\e ^{\epsilon_\rho} \cdot \left( \dfrac{(\deg(f) - q)\sigma ^{-\zeta} + qL_1^\zeta[1+(L_1-1)^\zeta] }{\deg(f)}\right) \approx \dfrac{2(3^{-\zeta}) + 1}{3} < 1.
\end{equation}The above relation shows that the system satisfies (\ref{kdljfhkdjfkasd}).

 Now we will prove that this system satisfies (\ref{f32}) of (f3). Note that,
\begin{equation}\label{dkjdk}
\rho (x,y) := \dfrac{1}{|\det Df (x,y)|} = \dfrac{1}{3g'(x)}. 
\end{equation}Besides that,  since $g$ is $\delta$-close to $\id$ in the $C^2$ topology, we have 
\begin{equation}\label{dkjdk}
-\delta \leq g(x) -x \leq \delta,
\end{equation}
\begin{equation}\label{dkjdk}
1-\delta \leq g'(x) \leq1+ \delta
\end{equation}and
\begin{equation}\label{dkjdk}
-\delta \leq g''(x) \leq \delta.
\end{equation}In what follows, the point $x_2$ is obtained by an application of the Mean Value Theorem. Then, we have

\begin{eqnarray*}
\dfrac{\left| \dfrac{1}{\left| \det Df(x_0,y_0)\right|}- \dfrac{1}{\left| \det Df(x_1,y_1)\right|}\right|}{d_2((x_0,y_0), (x_1,y_1))^\zeta}&=&\dfrac{\left| \dfrac{1}{\left| 3g'(x_0)\right|}- \dfrac{1}{\left|3g'(x_1)\right|}\right|}{\max \{d_1(x_0,x_1), d_1(y_0,y_1)\}^\zeta}\\&\leq&\dfrac{1}{3}\dfrac{\left| g'(x_1)-g'(x_0)\right|}{d_1(x_0,x_1)^\zeta}\dfrac{1}{|g'(x_1)g'(x_0)|}\\&\leq&\dfrac{1}{3}\left| g''(x_2)\right|\dfrac{1}{|g'(x_1)g'(x_0)|}\\&\leq&\dfrac{1}{3}\delta \dfrac{1}{(1-\delta)^2}=\dfrac{1}{3}\sqrt{\delta}\sqrt{\delta} \dfrac{1}{(1-\delta)^2}\\&\leq&\dfrac{1}{3}\dfrac{1}{1+\delta}\sqrt{\delta} \dfrac{1}{(1-\delta)^2}; \ \text{for small} \ \delta \\&\leq&\dfrac{\sqrt{\delta}}{(1-\delta)^2} \inf_{x\in [0,1]} \dfrac{1}{3g'(x)}.
\end{eqnarray*}Thus,

\begin{equation}\label{jhsdgjfsa}
H_\zeta (\rho) \leq \dfrac{\sqrt{\delta}}{(1-\delta)^2} \inf_{x\in [0,1]} \rho.
\end{equation}If $\delta$ is small enough, by equations (\ref{dfdfds}) and (\ref{jhsdgjfsa}), the perturbed system satisfies (\ref{f31}) and (\ref{f32}) of (f3).

We emphasize that this example satisfies the hypotheses of both articles, \cite{VAC} and \cite{VM}. More precisely, it satisfies (H1), (H2) and (P) of \cite{VAC} and (H1), (H2) and (P) of \cite{VM}. In fact, by the variational principle, we have that $h(f)>0$.  
\end{example}

\begin{example}\label{sesprowerpo} Now we present a general idea to generate examples by perturbing the identity or an expanding map close to the identity.
	
	Let $f_0: \mathbb{T}^d \longrightarrow \mathbb{T}^d$ be an expanding map. Choose a covering $\mathcal{P}$ and an atom $P_1 \in \mathcal{P}$ that contains a periodic point (possibly a fixed point) $p$. Next, consider a perturbation $f$, of $f_0$, within $P_1$ using a pitchfork bifurcation in such a way that $p$ becomes a saddle point for $f$. Consequently, $f$ coincides with $f_0$ in $P_1^c$, where we have uniform expansion. The perturbation can be designed to satisfy condition (f1). This ensures that $f$ is never overly contracting in $P_1$, and it remains topologically mixing. However, it is important to note that a small perturbation with these properties may not always exist. If such a perturbation does exist, then condition (f3) can be satisfied. In this case, $m_1$ is absolutely continuous with respect to the Lebesgue measure, which is an expanding conformal and positive measure on open sets. Consequently, there can be no periodic attractors.
\end{example}
\begin{example}\label{sesprowerpoo}
In the previous example, assume that $f_0$ is diagonalizable, with eigenvalues $1< 1+ a<\lambda$, associated with $e_1, e_2$, respectively, where $x_0$ is a fixed point. Fix $a,\epsilon> 0$ such that $\log(\frac{1+ a}{1- a})<\epsilon$ and
\begin{equation*}
	\e ^\epsilon\left(\frac{(\deg(f_0)- 1)(1+ a)^{-\zeta}+ (1/(1- a))^{\zeta}[1+ (a/(1- a))^\zeta]}{\deg(f_0)}\right)< 1.
\end{equation*}

Note that any smaller $a> 0$ will still satisfy these equations.

Let $\mathcal U$ be a finite covering of $M$ by open domains of injectivity for $f_0$. Redefining sets in $\mathcal U$, we may assume $x_0= (m_0, n_0)$ belongs to exactly one such domain $U$. Let $r> 0$ be small enough that $B_{2r}(x_0)\subset U$. Define $\rho=\eta_r\ast g$, where $\eta_r(z)= (1/r^2)\eta(z/r)$, $\eta$ denotes the standard mollifier, and
\begin{equation*}
	g(m,n)=\begin{cases}
		\lambda(1- a),&\text{if }(m, n)\in B_r(x_0);\\
		\lambda(1+ a),&\text{otherwise.}
	\end{cases}
\end{equation*}
Finally, define a perturbation $f$ of $f_0$ by
\begin{equation*}
	f(m, n)= (m_0+\lambda (m- m_0),n_ 0+(\rho(m, n)/\lambda) (n- n_0)).
\end{equation*}
Then $x_0$ is a saddle point of $f$ and the desired conditions are satisfied for $\mathcal A= B_{2r}(x_0)$, $L_1= 1/(1- a)$ and $\sigma= 1+ 2a$. The only non-trivial condition is (f3). To show it, note that

\begin{equation*}
	\rho(x)-\rho(y)=\int_S \frac{2a}{\lambda(1- a^2)}\eta_r(z)\,dz-\int_{S'}\frac{2a}{\lambda(1- a^2)}\eta_r(z)\,dz,
\end{equation*}
where $S=\{z\in\mathbb R^2: x- z\in B_r(x_ 0), y- z\notin B_r(x_ 0)\}$ and $S'=\{z\in\mathbb R^2: y- z\in B_r(x_ 0), x- z\notin B_r(x_ 0)\}$. Take $x, y\in\mathbb R^2$ and write $|x- y|= qr$, $A_q=\{z\in\mathbb R^2: 1- q< |z|< 1\}$. We have
\begin{equation*}
	\frac{|\rho(x)-\rho(y)|}{|x- y|^\zeta}\leq\frac{2a\eta_r(S)}{\lambda(1- a^2)q^\zeta r^\zeta}\leq\frac{2a\eta(A_q)/q^\zeta}{\lambda(1- a^2)}.
\end{equation*}
Since $N=\sup_{q> 0}\eta(A_q)/q^\zeta< +\infty$, we can take $a$ so small that $2aN/(1- a)<\epsilon$, therefore $H_\zeta(\rho)<\epsilon\inf\rho$.
\end{example}

\begin{example}{(Discontinuous Maps)} Let $F=(f, G)$ be the measurable map, where $f$ is from the previous Example \ref{new}. Consider the real numbers $\alpha _1$ and $\alpha_2$ s.t $ 0\leq \alpha _1 < \alpha _2 < 1$. Let  $G:[0,1] \times [0,1] \longrightarrow [0,1]$ be the function defined by
		$$
	G(x,y)=\left\{\begin{array}{ll}
	\alpha_1 y & \text { if } 0 \leqslant x \leqslant \frac{1}{2}, \\
		\alpha_2 y & \text { if } \frac{1}{2}<x \leqslant 1.
	\end{array}\right.
	$$It is easy to see that $G$ is discontinuous on the set $\{\frac{1}{2}\} \times [0,1] $. Moreover, $G$ satisfies (H2) since $|G|_\zeta =0$ (see equation (\ref{jdhfjdh})), for all $\zeta$. Thus, $G$ is a $\alpha _3$-contraction, where $\alpha_3 = \max \{\alpha_1, \alpha _2\}$. Since $L_1=1$ we have that $(\alpha_3L_1)^\zeta <1$, for all $\zeta$. Therefore, $F$ satisfies all hypothesis (f1), (f2), (f3), (G1), (G2) and $(\alpha_3L_1)^\zeta <1$.

\end{example}

\begin{example}{(Discontinuous Maps)} Let $F=(f, G)$ be the measurable map, where $f$ is again from the previous Example \ref{new}. Consider a real number $0 \leq \alpha_2 < 1$ and $\zeta$-H\"older functions $h_1:[0,\frac{1}{2}] \longrightarrow [0,1]$, $h_2:[\frac{1}{2},1] \longrightarrow [0,1]$ such that $h_1(\frac{1}{2}) \neq h_2(\frac{1}{2}) $ and $ 0\leq h_1, h_2 < \alpha _2 < 1$. Let  $G:[0,1] \times [0,1] \longrightarrow [0,1]$ be the function defined by
	$$
	G(x,y)=\left\{\begin{array}{ll}
		h_1(x) y & \text { if } 0 \leqslant x \leqslant \frac{1}{2}, \\
		h_2(x) y & \text { if } \frac{1}{2}<x \leqslant 1.
	\end{array}\right.
	$$It is easy to see that $G$ is discontinuous on the set $\{\frac{1}{2}\} \times [0,1] $. Moreover, $G$ satisfies (H2) since $|G|_\zeta \leq \max \{ |h_1|_\zeta, |h_2|_\zeta\}$. Thus, $G$ is a $\alpha _2$-contraction. Since $L_1=1$ we have that $(\alpha_2L_1)^\zeta <1$. Therefore, $F$ satisfies all hypothesis (f1), (f2), (f3), (G1), (G2) and $(\alpha_2L_1)^\zeta <1$.

\end{example}

While Theorem \ref{htyttigui} deals with a linear $R(\delta)$, this function can take other forms, as shown in the next example.
\begin{example}
Let us consider $F:M\times [0,1] \longrightarrow M\times [0,1]$ such that $F(x,y) = (f(x), G(x,y))$, where $G(x,y) = \lambda y$ for all $(x,y) \in M\times [0,1]$ and $0 < \lambda <1$. Suppose that $\delta_0$ is small enough in a way that $0 < \lambda + \sqrt{\delta}<1$ for all $\delta \in (0, \delta _0]$. Define $\{F_\delta\}_{\delta \in [0,1)}$ by $f_\delta := f$ for all $\delta \in [0,1)$, $G_\delta (x,y) = \lambda y$ for all $(x,y) \in M\times [0,1]$ if $\delta > \delta _0$ and $G_\delta(x,y) = (\sqrt{\delta} +\lambda)y$ for all $(x,y) \in M\times [0,1]$ if $\delta \in (0, \delta _0]$. 

We have that $\{F_\delta\}_{\delta \in [0,1)}$ is an \textbf{$R(\delta)$-perturbation} with $R(\delta) := \sqrt{\delta}$. Indeed,

\begin{eqnarray*}
 |G(x,y) - G_\delta(x,y)| &=&| \lambda y -  (\sqrt{\delta} +\lambda)y| 
 \\ &\leq&|\sqrt{\delta}y|
 \\ &\leq&\sqrt{\delta},
\end{eqnarray*}for all $\delta \in [0,1)$. Thus (U2.3) is satisfied. The other conditions are straightforward to check. 
\end{example}

\section{Preliminary Results}\label{seccc}

In this section, we present some preliminary and well-established results and introduce the functional analytic framework suitable for our approach. Some of these results are taken from \cite{RRR} and \cite{GLU}.

\subsection{Weak and Strong Spaces}
\subsubsection{$L^{\infty}$-like spaces.}\label{spa}

In this subsection, we define the vector spaces of signed measures that we will be working with. Specifically, we define the norm of the left-hand side of equation (\ref{stabll}). To do this, we need to briefly review some facts about the disintegration of measures, state Rokhlin's Disintegration Theorem, and establish certain notations.

\subsubsection*{Rokhlin's Disintegration Theorem}

Consider a probability space $(\Sigma,\mathcal{B}, \mu)$ and a partition $%
\Gamma$ of $\Sigma$ into measurable sets $\gamma \in \mathcal{B}$. Denote by $%
\pi : \Sigma \longrightarrow \Gamma$ the projection that associates to each
point $x \in M$ the element $\gamma _x$ of $\Gamma$ that contains $x$. That is, 
$\pi(x) = \gamma _x$. Let $\widehat{\mathcal{B}}$ be the $\sigma$-algebra of 
$\Gamma$ provided by $\pi$. Precisely, a subset $\mathcal{Q} \subset \Gamma$
is measurable if, and only if, $\pi^{-1}(\mathcal{Q}) \in \mathcal{B}$. We
define the \textit{quotient} measure $\mu _x$ on $\Gamma$ by $\mu _x(%
\mathcal{Q})= \mu(\pi ^{-1}(\mathcal{Q}))$.

The proof of the following theorem can be found in \cite{Kva}, Theorem
5.1.11 (items a), b) and c)) and Proposition 5.1.7 (item d)).

\begin{theorem}
	(Rokhlin's Disintegration Theorem) Suppose that $\Sigma $ is a complete and
	separable metric space, $\Gamma $ is a measurable partition of $\Sigma $ and $\mu $ is a probability on $\Sigma $. Then, $\mu $ admits a
	disintegration relative to $\Gamma $. That is, there exists a family $\{\mu _{\gamma}\}_{\gamma \in \Gamma }$ of probabilities on $\Sigma $ and a quotient measure $\mu _{x}$, such that:
	
	\begin{enumerate}
		\item[(a)] $\mu _\gamma (\gamma)=1$ for $\mu _x$-a.e. $\gamma \in \Gamma$;
		
		\item[(b)] for all measurable set $E\subset \Sigma $ the function $\Gamma
		\longrightarrow \mathbb{R}$ defined by $\gamma \longmapsto \mu _{\gamma
		}(E), $ is measurable;
		
		\item[(c)] for all measurable set $E\subset \Sigma $, it holds $\mu (E)=\int 
		{\mu _{\gamma }(E)}d\mu _{x}(\gamma )$.

	\label{rok}
	\item [(d)] If the $\sigma $-algebra $\mathcal{B}$ on $\Sigma $ has a countable
	generator, then the disintegration is unique in the following sense. If $(\{\mu _{\gamma }^{\prime }\}_{\gamma \in \Gamma },\mu _{x})$ is another disintegration of the measure $\mu $ relative to $\Gamma $, then $\mu
	_{\gamma }=\mu _{\gamma }^{\prime }$, for $\mu _{x}$-almost every $\gamma
	\in \Gamma $.
	\end{enumerate}
\end{theorem}

\subsubsection{The $\mathcal{L}^{\infty}$ and $S^\infty$ spaces}\label{jdfjdhkjf}

Let $\mathcal{SB}(\Sigma )$ be the space of Borel signed measures on $\Sigma : = M \times K$. Given $\mu \in \mathcal{SB}(\Sigma )$, denote by $\mu ^{+}$ and $\mu ^{-}$
the positive and the negative parts of its Jordan decomposition, $\mu =\mu
^{+}-\mu ^{-}$ (see remark {\ref{ghtyhh}). Let $\pi _{1}:\Sigma \longrightarrow M$ be the projection defined by $\pi_x (x,y)=x$, denote by $\pi _{1\ast }:$}$\mathcal{SB}(\Sigma )\rightarrow \mathcal{SB}(M)${\ the pushforward map associated to $\pi _{1}$. Denote by $\mathcal{AB}$ the set of signed measures $\mu \in \mathcal{SB}(\Sigma )$ such that its associated positive and negative marginal measures, $\pi _{1\ast }\mu ^{+}$ and $\pi _{1\ast }\mu ^{-},$ are absolutely continuous with respect to $m_{1}$. That is,
\begin{equation*}
	\mathcal{AB}=\{\mu \in \mathcal{SB}(\Sigma ):\pi _{1\ast }\mu ^{+}<<m_{1}\ \ 
	\mathnormal{and}\ \ \pi _{1\ast }\mu ^{-}<<m_{1}\}.  \label{thespace1}
\end{equation*}%
}Given a \emph{probability measure} $\mu \in \mathcal{AB}$ on $\Sigma$,
Theorem \ref{rok} describes a disintegration $\left( \{\mu _{\gamma
}\}_{\gamma },\mu _{x}\right)$ along $\mathcal{F}^{s}$ by a family of probability measures $\{\mu _{\gamma }\}_{\gamma }$, defined on the stable leaves.  Moreover, since 
$\mu \in \mathcal{AB}$, $\mu _{x}$ can be identified with a non-negative
marginal density $\phi _{1}:M\longrightarrow \mathbb{R}$, defined almost
everywhere, where $|\phi _{1}|_{1}=1$. For a non-normalized
positive measure $\mu \in \mathcal{AB}$ we can define its disintegration following the same idea.  In this case, $\{ \mu _{\gamma } \}$ is still a family of probability measures, $\phi _{1}$ is still defined and $|\phi _{1}|_{1}=\mu (\Sigma )$.

\begin{definition}
	Let $\pi _{2}:\Sigma \longrightarrow K$ be the projection defined by $
	\pi _{2}(x,y)=y$. Consider $\pi
	_{\gamma ,2}:\gamma \longrightarrow K$, the restriction of the map $\pi
	_{2}$ to the vertical leaf $\gamma $, and the
	associated pushforward map $\pi _{\gamma ,2\ast }$. Given a positive measure 
	$\mu \in \mathcal{AB}$ and its disintegration along the stable leaves $%
	\mathcal{F}^{s}$, $\left( \{\mu _{\gamma }\}_{\gamma },\mu _{x}=\phi
	_{1}m_{1}\right) $, we define the \textbf{restriction of $\mu $ on $\gamma $}
	and denote it by $\mu |_{\gamma }$ as the positive measure on $K$ (not
	on the leaf $\gamma $) defined, for all mensurable set $A\subset K$, as 
	\begin{equation*}
	\mu |_{\gamma }(A)=\pi _{\gamma ,2\ast }(\phi _{1}(\gamma )\mu _{\gamma
	})(A).
	\end{equation*}%
	For a given signed measure $\mu \in \mathcal{AB}$ and its Jordan
	decomposition $\mu =\mu ^{+}-\mu ^{-}$, define the \textbf{restriction of $%
		\mu $ on $\gamma $} by%
	\begin{equation*}
	\mu |_{\gamma }=\mu ^{+}|_{\gamma }-\mu ^{-}|_{\gamma }.
	\end{equation*}%
	\label{restrictionmeasure}
\end{definition}

\begin{remark}
	\label{ghtyhh}As proved in Appendix 2 of \cite {GLU},  restriction $%
	\mu |_{\gamma }$ does not depend on decomposition. Precisely, if $\mu
	=\mu _{1}-\mu _{2}$, where $\mu _{1}$ and $\mu _{2}$ are any positive
	measures, then $\mu |_{\gamma }=\mu _{1}|_{\gamma }-\mu _{2}|_{\gamma }$ $%
	m_{1}$-a.e. $\gamma \in M$. 
\end{remark}

Let $(X,d)$ be a compact metric space, $g:X\longrightarrow \mathbb{R}$ be a
$\zeta$-H\"older function, and $H_\zeta(g)$ be its best $\zeta$-H\"older's constant. That is,
\begin{equation}\label{lipsc}
\displaystyle{H_\zeta(g)=\sup_{x,y\in X,x\neq y}\left\{ \dfrac{|g(x)-g(y)|}{d(x,y)^\zeta}%
	\right\} }.
\end{equation}In what follows, we present a generalization of the Wasserstein-Kantorovich-like metric given in \cite{GLU} and \cite{GP}. 
\begin{definition}
	Given two signed measures, $\mu $ and $\nu $ on $X,$ we define the \textbf{\
		Wasserstein-Kantorovich-like} distance between $\mu $ and $\nu $ by 
	\begin{equation*}
	W_{1}^{\zeta}(\mu ,\nu ):=\sup_{H_\zeta(g)\leq 1,|g|_{\infty }\leq 1}\left\vert \int {\
		g}d\mu -\int {g}d\nu \right\vert .
	\end{equation*}%
	\label{wasserstein}
\end{definition}Since  $\zeta$ is a constant,  we denote%

\begin{equation}
||\mu ||_{W}:=W_{1}^{\zeta}(0,\mu ),  \label{WW}
\end{equation}and observe that $||\cdot ||_{W}$ defines a norm on the vector space of signed measures defined on a compact metric space. It is worth remarking that this norm is equivalent to the standard norm of the dual space of $\zeta$-H\"older functions.

\begin{definition}\label{sdfsdfsdasd}
	Let $\mathcal{L}^{\infty }\subseteq \mathcal{AB}(\Sigma )$ be the set of signed measures defined as%
	\begin{equation*}
	\mathcal{L}^{\infty }=\left\{ \mu \in \mathcal{AB}:\esssup ({W_{1}^{\zeta}(\mu
		^{+}|_{\gamma },\mu ^{-}|_{\gamma }))}<\infty \right\},
	\end{equation*}%
	where the essential supremum is taken over $M$ with respect to $m_{1}$.
	Define the function $||\cdot ||_{\infty }:\mathcal{L}^{\infty
	}\longrightarrow \mathbb{R}$ by%
	\begin{equation*}
	||\mu ||_{\infty }=\esssup ({W_{1}^{\zeta}(\mu ^{+}|_{\gamma },\mu ^{-}|_{\gamma
		}))}.
	\end{equation*}Finally, consider the following set of signed measures on $\Sigma $%
\begin{equation}\label{sinfi}
S^{\infty }=\left\{ \mu \in \mathcal{L}^{\infty };\phi _{1}\in
H_\zeta \right\},
\end{equation}%
and the function, $||\cdot ||_{S^{\infty }}:S^{\infty }\longrightarrow 
\mathbb{R}$, defined by%
\begin{equation*}
||\mu ||_{S^{\infty }}=|\phi _{1}|_{\zeta}+||\mu ||_{\infty }.
\end{equation*}
\end{definition}

\begin{remark}
	A straightforward computation yields $||\cdot ||_W \leq ||\cdot||_\infty$. Then, supposing that $\{F_\delta\}_{\delta \in [0,1)}$ satisfies Theorem (\ref{htyttigu}), it holds $$||\mu_{\delta }-\mu_{0}||_{W}\leq AR(\delta)^\zeta \log \delta ,$$for some $A>0$. Therefore, for all $\zeta$-Holder function $g:\Sigma \longrightarrow \mathbb{R}$, the following estimate holds $$\left|\int{g}d\mu_\delta - \int{g}d\mu_0\right| \leq A ||g||_{\zeta} R(\delta) ^\zeta \log \delta,$$where $||g||_{\zeta} = ||g||_\infty + H_\zeta(g)$ (see equation (\ref{lipsc}), for the definition of $H_\zeta(g)$). Thus, for all $\zeta$-Holder function, $g:\Sigma \longrightarrow\mathbb{R}$, the limit $\displaystyle{\lim _{\delta \longrightarrow 0} {\int{g}d\mu_\delta} = \int{g}d\mu_0}$ holds, with a rate of convergence smaller than or equal to $R(\delta)^\zeta \log \delta$.
\end{remark}

The proof of the next proposition is straightforward and can be found in 
\cite{L}.

\begin{proposition}
	$\left( \mathcal{L}^{\infty },||\cdot ||_{\infty }\right) $ and $\left(
	S^{\infty },||\cdot||_{S^{\infty }}\right) $ are normed vector spaces.
\end{proposition}

\subsection{The transfer operator associated to $F$}

In this section, we examine the transfer operator associated with skew-product maps, $F=(f, G)$, as defined in Section \ref{intro}. We analyze its action on our disintegrated measure spaces, $\mathcal{L}^\infty$ and $S^\infty$, which were introduced in Section \ref{jdfjdhkjf}. For the transfer operator applied to measures, a type of Perron-Frobenius formula holds (see Corollary \ref{oierew}). This formula bears some resemblance to the one that applies to one-dimensional maps.

Consider the pushforward map (also known as the "transfer operator") $\func{F}_{\ast }$ associated with $F$, defined
by 
\begin{equation*}
\lbrack \func{F}_{\ast }\mu ](E)=\mu (F^{-1}(E)),
\end{equation*}%
for each signed measure $\mu \in \mathcal{SB}(\Sigma )$ and for all
measurable set $E\subset \Sigma $, where $\Sigma:=M\times K$.

The reader can find the proofs of the following three results in Lemma 4.1, Proposition 2, and Corollary 2 of \cite{RRR}, respectively.

\begin{lemma}
	\label{transformula}For every probability $\mu \in \mathcal{AB}$ disintegrated
	by $(\{\mu _{\gamma }\}_{\gamma },\phi _1)$, the disintegration $(\{(\func{%
		F}_{\ast }\mu )_{\gamma }\}_{\gamma },(\func{F}_{\ast }\mu )_{x})$ of the
	pushforward $\func{F}_{\ast }\mu $ \ satisfies the following relations%
	\begin{equation}
	(\func{F}_{\ast }\mu )_{x}=\func{P}_{f}(\phi _1)m_{1}  \label{1}
	\end{equation}
	and
	\begin{equation}
	(\func{F}_{\ast }\mu )_{\gamma }=\nu _{\gamma }:=\frac{1}{\func{P}_{f}(\phi
		_1)(\gamma )}\sum_{i=1}^{\deg(f)}{\frac{\phi _1}{|\det Df_{i}|}\circ
		f_{i}^{-1}(\gamma )\cdot \chi _{f_{i}(P_{i})}(\gamma )\cdot \func{F}_{\ast
		}\mu _{f_{i}^{-1}(\gamma )}}  \label{2}
	\end{equation}
	when $\func{P}_{f}(\phi _x)(\gamma )\neq 0$. Otherwise, if $\func{P}%
	_{f}(\phi _1)(\gamma )=0$, then $\nu _{\gamma }$ is the Lebesgue\footnote{There is nothing special about the Lebesgue measure here. We could replace it with any other positive measure.} measure
	on $\gamma $ (the expression $\displaystyle{\frac{\phi _1}{|\det Df_{i}|}%
		\circ f_{i}^{-1}(\gamma )\cdot \frac{\chi _{f_{i}(P_{i})}(\gamma )}{\func{P}%
			_{f}(\phi _1)(\gamma )}\cdot \func{F}_{\ast }\mu _{f_{i}^{-1}(\gamma )}}$
	is understood to be zero outside $f_{i}(P_{i})$ for all $i=1,\cdots ,\deg(f)$).
	Here and above, $\chi _{A}$ is the characteristic function of the set $A$.
\end{lemma}

\begin{proposition}
	\label{niceformulaab}Let $\gamma \in \mathcal{F}^{s}$ be a stable leaf. Let
	us define the map $F_{\gamma }:K\longrightarrow K$ by 
	\begin{equation}\label{ritiruwt}
	F_{\gamma }=\pi _{2}\circ F|_{\gamma }\circ \pi _{\gamma ,2}^{-1}.
	\end{equation}%
	Then, for each $\mu \in \mathcal{L}^{\infty}$ and for almost all $\gamma \in
	M$ it holds 
	\begin{equation}
	(\func{F}_{\ast }\mu )|_{\gamma }=\sum_{i=1}^{\deg(f)}{\func{F}%
		_{\gamma _i \ast }\mu |_{\gamma _i }\rho _i(\gamma _i)\chi _{f_{i}(P_{i})}(\gamma )}\ \ m_{1}%
	\mathnormal{-a.e.}\ \ \gamma \in M  \label{niceformulaa}
	\end{equation}%
	where $\func{F}_{\gamma_i \ast }$ is the pushforward map
	associated to $\func{F}_{\gamma_i}$, $\gamma _i = f_{i}^{-1}(\gamma )$ when $\gamma \in f_i (P_i)$ and $\rho_i(\gamma)= \dfrac{1}{|\det (f_i^{'}(\gamma))|}$, where $f_i = f|_{P_i}$.
\end{proposition}

\begin{remark}\label{chkjg} 
	By (f2),  (see \cite{RRR}) there exists a disjoint finite family, $\mathcal{P}$, of open sets, $P_1, \cdots, P_{\deg{(f)}}$, s.t. $\bigcup_{i=1}^{\deg{(f)}} P_i=M$ $m_1$-a.e., and $f|_{P_{i}}:P_i \longrightarrow f(P_i)$ is a diffeomorfism for all $i=1, \cdots \deg{(f)}$. Moreover, $f(P_i)=M$ $m_1$-a.e., for all $i=1, \cdots, \deg(f)$. Therefore, it holds that $$\func{P}_f(\varphi)(x) = \sum _{i=1}^{\deg{(f)}} {\varphi (x_i)\rho (x_i)},$$ for $m_1$-a.e. $x \in M$, where $$\rho_i(\gamma):= \frac{1}{|\det (f_i^{'}(\gamma))|}$$ and $f_i = f|_{P_i}$. This expression will be used later on.
\end{remark}

Sometimes it will be convenient to use the following expression for $(\func{F}_{\ast } \mu )|_{\gamma }$, which is a consequence of Remark \ref{chkjg} and Proposition \ref{niceformulaab}.
\begin{corollary}\label{oierew}
	For each $\mu \in \mathcal{L}^{\infty}$ it holds 
	\begin{equation}
	(\func{F}_{\ast }\mu )|_{\gamma }=\sum_{i=1}^{\deg(f)}{\func{F}%
		_{\gamma _i \ast }\mu |_{\gamma _i }\rho _i(\gamma _i)}\ \ m_{1}%
	\mathnormal{-a.e.}\ \ \gamma \in M,  \label{niceformulaaareer}
	\end{equation}%
	where $\gamma _i$ is the $i$-th pre image of $\gamma$, $i=1,\cdots, \deg(f)$.
\end{corollary}

\subsection{Basic properties of the norms and convergence to equilibrium}

In this part, we list the properties of the norms and their behavior concerning the action of the transfer operator.

According to \cite{VAC} and \cite{VM}, a map $f: M \longrightarrow M$ satisfying (f1), (f2), and (f3) has an invariant probability measure $m_1$ of maximal entropy. The Perron-Frobenius operator of $f$, denoted as $\func{P}_f: L^1_{m_1} \longrightarrow L^1_{m_1}$, satisfies the following two results, the proofs of which can be found in \cite{RRR}.

\begin{theorem}\label{loiub}
	There exist $0< r<1$ and $D>0$ such that for all $\varphi \in H_\zeta$,  and $\int{\varphi}dm_1 =0$, it holds $$|\func{P_f}^n(\varphi)|_\zeta \leq Dr^n|\varphi|_\zeta \ \ \forall \ n \geq 1,$$where $|\varphi|_\zeta := H_\zeta (\varphi) + |\varphi|_{\infty}$.  
\end{theorem}

\begin{theorem}\label{asewqtw} (Lasota-Yorke inequality) There exist $k\in \mathbb{N}$, $%
	0<\beta _{0}<1$ and $C>0$ such that, for all $g\in H_\zeta$, it holds 
	\begin{equation}
	|\func{P}_{f}^{k}g|_{\zeta}\leq \beta _{0}|g|_{\zeta}+C|g|_{\infty},  \label{LY1}
	\end{equation}where $|g|_\zeta := H_\zeta (g) + |g|_{\infty}$.
	
\end{theorem}

\begin{corollary}\label{irytrtrte}
	There exist constants $B_3>0$, $C_2>0$ and $0<\lambda<1$ such that for all $%
	g \in H_\zeta$, and all $n \geq 1$, it holds
	
	\begin{equation}
	|\func{P}_{f}^{n}g|_{\zeta} \leq B_3 \lambda ^n | g|_{\zeta} + C_2|g|_{\infty}.
	\label{lasotaiiii}
	\end{equation}
\end{corollary}

In the following, item (1) demonstrates the continuity and weak contraction of the transfer operator, $\func {F}_*$, with respect to the norm $||\cdot||_\infty$. Items (2) and (3) provide Lasota-Yorke inequalities for the norms $||\cdot||_\infty$ and $||\cdot||_{S^\infty}$, showing a regularizing property of the transfer operator for these norms. These inequalities are also commonly referred to as Doeblin-Fortet inequalities. The proofs of equations (\ref{weakcontral11234}), (\ref{lasotaoscilation2}), (\ref{nicecoro}), and (\ref{nicecoroo}) can be found in Proposition 3, Proposition 4, Corollary 3, and Lemma 5.2 of \cite{RRR}, respectively.

\begin{enumerate}
	\item (Weak Contraction for $||\cdot||_\infty$) If $\mu \in \mathcal{L}^{\infty}$, then 
	\begin{equation}\label{weakcontral11234}
	||\func{F}_{\ast }\mu ||_{\infty}\leq ||\mu ||_{\infty};
	\end{equation}%

	\item (Lasota-Yorke inequality for $S^{\infty}$) 
	There exist $A$, $B_{2}>0$ and $\lambda <1$ ($\lambda$ of Corollary \ref{irytrtrte}) such that, for all $\mu
	\in S^{\infty}$, it holds
	\begin{equation}
	||\func{F}_{\ast }^{n}\mu ||_{S^{\infty}}\leq A\lambda ^{n}||\mu
	||_{S^{\infty}}+B_{2}||\mu ||_{\infty},\ \ \forall n\geq 1;  \label{lasotaoscilation2}
	\end{equation}%
	
	\item 	For every signed measure $\mu \in \mathcal{L}^{\infty}$, it holds 
	\begin{equation}\label{nicecoro}
	||\func{F}_{\ast }^{n}\mu ||_{\infty}\leq (\alpha ^\zeta) ^{n}||\mu ||_{\infty}+\overline{%
		\alpha }|\phi _{1}|_{\infty},
	\end{equation}%
	where $\overline{\alpha }=\frac{1 }{1-\alpha^\zeta }$;
	
		\item 	For every signed measure $\mu$ on $K$, such that $\mu(K)=0$ it holds 
	\begin{equation}\label{nicecoroo}
	||\func{F}_{\gamma *} \mu ||_{W}\leq \alpha ^\zeta ||\mu ||_{W},
	\end{equation}where $F_\gamma$ is defined in equation (\ref{ritiruwt}).
	
\end{enumerate}

\subsection{Convergence to equilibrium}\label{invt}

Let $X$ be a compact metric space. Consider the space $\mathcal{SB}(X)$ of
Boreleans signed measures on $X$ and two normed vectors subspaces, $(B_{s},||~||_{s})\subseteq (B_{w},||~||_{w})\subseteq \mathcal{SB}(X)$ with
norms satisfying
\begin{equation*}
||~||_{w}\leq ||~||_{s}.
\end{equation*}%
We say that a Markov operator
\begin{equation*}
	\text{L}:B_{w}\longrightarrow B_{w}
\end{equation*}
has convergence to the equilibrium with a speed of at least $\Phi $ for
the norms $||\cdot ||_{s}$ and $||\cdot ||_{w}$, if for each $\mu \in 
\mathcal{V}_{s}$, where 
\begin{equation}
\mathcal{V}_{s}=\{\mu \in B_{s},\mu (X)=0\}  \label{vs}
\end{equation}%
is the space of zero-average measures, it holds 
\begin{equation*}
||\text{L}^{n}(\mu )||_{w}\leq \Phi (n)||\mu ||_{s},  \label{wwe}
\end{equation*}%
where $\Phi (n)\longrightarrow 0$ as $n\longrightarrow \infty $.

Let us consider the set of zero average measures in $S^{\infty}$ defined by 
\begin{equation}
\mathcal{V}_{s}=\{\mu \in S^{\infty}:\mu (\Sigma )=0\}.  \label{mathV}
\end{equation}%

The proof of the next proposition can be found in \cite[Proposition 6]{RRR}.
\begin{theorem}[Exponential convergence to equilibrium]
	\label{5.8} There exist $D_{2}\in \mathbb{R}$ and $0<\beta _{1}<1$ such that
	for every signed measure $\mu \in \mathcal{V}_{s}$, it holds 
	\begin{equation*}
	||\func{F}_{\ast }^{n}\mu ||_{\infty}\leq D_{2}\beta _{1}^{n}||\mu ||_{S^{\infty}},
	\end{equation*}%
	for all $n\geq 1$, where $\beta _{1}=\max \{\sqrt{r},%
	\sqrt{\alpha^\zeta }\}$ and $D_{2}=(\sqrt{\alpha^\zeta }^{-1}+\overline{\alpha }D \sqrt{r}^{-1})$.\label{quasiquasiquasi}
\end{theorem}

\subsection{H\"older-Measures}

In this section, we introduce in Definition \ref{Lips3} the concept of Holder's constant of a signed measure on $\Sigma$. We also make use of the hypotheses (G2) for the first time. Moreover, apart from satisfying equation (\ref{kdljfhkdjfkasd}), the constant $L_1$ mentioned in (f1) and (f3) is also required to be sufficiently close to $1$ such that $(\alpha \cdot L_1)^\zeta<1$ (or $\alpha$ is close enough to $0$). This condition is satisfied by the examples of Section \ref{ooisidrosr}.

We have observed that a positive measure on $M \times K$ can be disintegrated along the stable leaves $\mathcal{F}^s$ in such a way that we can regard it as a family of positive measures on $M$, denoted by $\{\mu |_\gamma\}_{\gamma \in \mathcal{F}^s}$. Since there exists a one-to-one correspondence between $\mathcal{F}^s$ and $M$, this defines a path in the metric space of positive measures ($\mathcal{SB}(K)$) defined on $K$, represented by $M \longmapsto \mathcal{SB}(K)$, where $\mathcal{SB}(K)$ is equipped with the Wasserstein-Kantorovich-like metric (see Definition \ref{wasserstein}).

It will be convenient to use functional notation and denote such a path by $\Gamma_{\mu}: M \longrightarrow \mathcal{SB}(K)$, defined almost everywhere by $\Gamma_{\mu}(\gamma) = \mu|_\gamma$, where $(\{\mu_{\gamma}\}_{\gamma \in M}, \phi_{1})$ is some disintegration of $\mu$.
However, since this disintegration is defined $\widehat{\mu}$-a.e. $\gamma \in M$, the path $\Gamma_\mu$ is not unique. For this reason, we define $\Gamma_{\mu}$ as the class of almost everywhere equivalent paths corresponding to $\mu$.

\begin{definition}
	Consider a positive Borelean measure $\mu$ on $M \times K$, and a disintegration  $\omega=(\{\mu _{\gamma }\}_{\gamma \in M},\phi
	_1)$, where $\{\mu _{\gamma }\}_{\gamma \in M }$ is a family of
	probabilities on $M \times K$ defined $\widehat{\mu}$-a.e. $\gamma \in M$ (where $\widehat{\mu} := \pi_1{_*}\mu=\phi _1 m_1$) and $\phi
	_1:M\longrightarrow \mathbb{R}$ is a non-negative marginal density. Denote by $\Gamma_{\mu }$ the class of equivalent paths associated to $\mu$ 
	\begin{equation*}
	\Gamma_{\mu }=\{ \Gamma^\omega_{\mu }\}_\omega,
	\end{equation*}
	where $\omega$ ranges on all the possible disintegrations of $\mu$ and $\Gamma^\omega_{\mu }: M\longrightarrow \mathcal{SB}(K)$ is the map associated to a given disintegration, $\omega$:
	$$\Gamma^\omega_{\mu }(\gamma )=\mu |_{\gamma } = \pi _{\gamma, 2} ^\ast \phi _1
	(\gamma)\mu _\gamma .$$
\end{definition}Let us call the set on which $\Gamma_{\mu }^\omega $ is defined by $I_{\Gamma_{\mu }^\omega } \left( \subset M\right)$.

\begin{definition}For a given $0<\zeta <1$, a disintegration $\omega$ of $\mu$, and its functional representation $\Gamma_{\mu }^\omega $, we define the \textbf{$\zeta$-H\"older constant of $\mu$ associated to $\omega$} by

	\begin{equation}\label{Lips1}
	|\mu|_\zeta ^\omega := \esssup _{\gamma_1, \gamma_2 \in I_{\Gamma_{\mu }^\omega}} \left\{ \dfrac{||\mu|_{\gamma _1}- \mu|_{\gamma _2}||_W}{d_1 (\gamma _1, \gamma _2)^\zeta}\right\}.
	\end{equation}Finally, we define the \textbf{$\zeta$-H\"older constant} of the positive measure $\mu$ by

	\begin{equation}
	\label{Lips2}
	|\mu|_\zeta :=\displaystyle{\inf_{ \Gamma_{\mu }^\omega \in \Gamma_{\mu } }\{|\mu|_\zeta ^\omega\}}.
	\end{equation}
	
	\label{Lips3}
\end{definition}

\begin{remark}
	When no confusion is possible, to simplify the notation, we denote $\Gamma_{\mu }^\omega (\gamma )$ just by $\mu |_{\gamma } $.
\end{remark}

\begin{definition}\label{sdfsdjhfjhsgjdfgjsd}
	From the Definition \ref{Lips3} we define the set of the $\zeta$-H\"older`s positive measures $\mathcal{H} _\zeta^{+}$ as
	\begin{equation}
	\mathcal{H} _\zeta^{+}=\{\mu \in \mathcal{AB}:\mu \geq 0,|\mu |_\zeta <\infty \}.
	\end{equation}
\end{definition}

For the next lemma, for a given path $\Gamma _\mu$ which represents the measure $\mu$, we define for each $\gamma \in I_{\Gamma_{\mu }^\omega }\subset M$, the map

\begin{equation}
\mu _F(\gamma) := \func{F_\gamma }_*\mu|_\gamma,
\end{equation}where $F_\gamma :K \longrightarrow K$ is defined as

\begin{equation}\label{poier}
F_\gamma (y) = \pi_2 \circ F \circ {(\pi _2|_\gamma)} ^{-1}(y)
\end{equation}and $\pi_2 : M\times K \longrightarrow  K$ is the second coordinate projection $\pi_2(x,y)=y$.

The proofs of Lemma \ref{apppoas}, Proposition \ref{iuaswdas} and Corollary \ref{kjdfhkkhfdjfh} can be found in Lemma 7.4, Proposition 9 and Corollary 4 of \cite{RRR}.

\begin{lemma}\label{apppoas}
	Suppose that $F:\Sigma \longrightarrow \Sigma$ satisfies (G1) and (G2). Then, for all $\mu \in \mathcal{H} _\zeta^{+} $ which satisfy $\phi _1 = 1$ $m_1$-a.e., it holds $$||\func{F}%
	_{x  \ast }\mu |_{x  } - \func{F}%
	_{y \ast }\mu |_{y  }||_W \leq \alpha^\zeta |\mu|_\zeta  d_1(x, y)^\zeta  + |G|_\zeta d_1(x, y)^\zeta ||\mu ||_\infty,$$ for all $x,y \in P_i$ and all $i=1, \cdots, \deg(f)$.
\end{lemma}

\begin{corollary}\label{nslfdflsdjlf}
	Let $\{F_{\delta }\}_{\delta \in [0,1)}$ an admissible $R(\delta)$-perturbation and $\gamma_{\delta, i}$ the $i$-th pre-image of $\gamma \in M$ by $f_\delta$, $i=1, \cdots, \deg(f_\delta)$. Then, for all $\mu \in \mathcal{H} _\zeta^{+} $ which satisfy $\phi _1 = 1$ $m_1$-a.e., the following inequality holds: $$\left\vert \left\vert ({\func{F}_{0,\gamma _{0,i} }{_\ast }}- \func{F}_{0,\gamma _{\delta,i} }{_\ast })\mu |_{\gamma _{0,i}}\right\vert \right\vert _{W} \leq R(\delta)^\zeta( 2 \alpha^\zeta |\mu|_\zeta + |G|_\zeta ||\mu ||_\infty ) , \forall i=1, \cdots, \deg(f),$$where $F_{\delta,\gamma _{\delta,i}}$ is defined by equation (\ref{poier}), for all $\delta \in [0,1)$.
\end{corollary}
\begin{proof}
	To simplify the notation, we denote $F:=F_0$ and $\gamma:=\gamma_{0,i}$. Thus, we have
	
	\begin{eqnarray*}
		\left\vert \left\vert ({\func{F}_{0,\gamma _{0,i} }{_\ast }}- \func{F}_{0,\gamma _{\delta,i} }{_\ast })\mu |_{\gamma _{0,i}}\right\vert \right\vert _{W} &=& \left\vert \left\vert ({\func{F}_{\gamma }{_\ast }}- \func{F}_{\gamma _{\delta,i} }{_\ast })\mu |_{\gamma}\right\vert \right\vert _{W} \\&=& \left\vert \left\vert {\func{F}_{\gamma }{_\ast }\mu |_{\gamma}}- \func{F}_{\gamma _{\delta,i} }{_\ast }\mu |_{\gamma}\right\vert \right\vert _{W} \\&\leq& \left\vert \left\vert  {\func{F}_{\gamma }{_\ast }\mu |_{\gamma}}- \func{F}_{\gamma _{\delta,i} }{_\ast }\mu |_{\gamma_{\delta,i}}\right\vert\right\vert _{W} +  \left\vert \left\vert \func{F}_{\gamma _{\delta,i} }{_\ast }(\mu |_{\gamma_{\delta,i}} - \mu |_{\gamma} ) \right\vert\right\vert _{W}
	\end{eqnarray*}Since $\phi _1 = 1$ $m_1$-a.e., $\mu |_{\gamma_{\delta,i}} - \mu |_{\gamma}$ has zero average. Therefore, by Lemma \ref{apppoas}, equation (\ref{nicecoroo}), (U2.2) and definition (\ref{Lips3}) applied on $\mu$, we get
	\begin{eqnarray*}
		\left\vert \left\vert ({\func{F}_{0,\gamma _{0,i} }{_\ast }}- \func{F}_{0,\gamma _{\delta,i} }{_\ast })\mu |_{\gamma _{0,i}}\right\vert \right\vert _{W} &\leq& \left\vert \left\vert  {\func{F}_{\gamma }{_\ast }\mu |_{\gamma}}- \func{F}_{\gamma _{\delta,i} }{_\ast }\mu |_{\gamma_{\delta,i}}\right\vert\right\vert _{W} +  \alpha ^\zeta \left\vert \left\vert \mu |_{\gamma_{\delta,i}} - \mu |_{\gamma}  \right\vert\right\vert _{W} \\&\leq& 	\alpha^\zeta |\mu|_\zeta  d_1(\gamma_{\delta,i}, \gamma)^\zeta  + |G|_\zeta d_1(\gamma_{\delta,i}, \gamma)^\zeta ||\mu ||_\infty \\&+&  \alpha ^\zeta | \mu |_\zeta d_1(\gamma_{\delta,i}, \gamma)^\zeta
		\\&\leq& R(\delta)^\zeta( 2 \alpha^\zeta |\mu|_\zeta + |G|_\zeta ||\mu ||_\infty ).
\end{eqnarray*}

\end{proof}
\begin{lemma}\label{çhjghljk}
	Let $\{F_{\delta }\}_{\delta \in [0,1)}$ an admissible $R(\delta)$-perturbation and $\gamma_{\delta, i}$ the $i$-th pre-image of $\gamma \in M$ by $f_\delta$, $i=1, \cdots, \deg(f_\delta)$. Then, the following inequality holds: $$\left\vert \left\vert ({\func{F}_{0,\gamma _{\delta,i} }{_\ast }}- \func{F}_{\delta,\gamma _{\delta,i} }{_\ast })\mu |_{\gamma _{0,i}}\right\vert \right\vert _{W} \leq ||\mu|_{\gamma_{0,i}}|| R(\delta)^\zeta, \forall i=1, \cdots, \deg(f),$$ where $F_{\delta,\gamma _{\delta,i}}$ is defined by equation (\ref{poier}), for all $\delta \in [0,1)$.
\end{lemma}

\begin{proof}
To simplify the notation, we denote $\gamma:=\gamma_{\delta,i}$. Thus, by definition (\ref{wasserstein}) and (U2.3), we have

\begin{eqnarray*}
	\left\vert \left\vert ({\func{F}_{0,\gamma }{_\ast }}- \func{F}_{\delta,\gamma }{_\ast })\mu |_{\gamma _{0,i}}\right\vert \right\vert _{W} &=& \left\vert \left\vert ({\func{F}_{0, \gamma }{_\ast }}- \func{F}_{\delta, \gamma}{_\ast })\mu |_{\gamma _{0,i}}\right\vert \right\vert _{W} \\&=& \sup _{H_\zeta(g)\leq 1,|g|_{\infty }\leq 1} \left\vert \int {g}  d ({\func{F}_{0, \gamma }{_\ast }}\mu |_{\gamma _{0,i}} - \func{F}_{\delta, \gamma}{_\ast }\mu |_{\gamma _{0,i}}) \right\vert  \\&=& \sup _{H_\zeta(g)\leq 1,|g|_{\infty }\leq 1} \left\vert \int {g (G_0(\gamma,y))  - g (G_\delta(\gamma,y))}  d \mu |_{\gamma _{0,i}} \right\vert \\&\leq & \sup _{H_\zeta(g)\leq 1,|g|_{\infty }\leq 1}  \int { \left\vert g (G_0(\gamma,y))  - g (G_\delta(\gamma,y) ) \right\vert}  d \mu |_{\gamma _{0,i}}  \\&\leq&   \int {d_2(G_0(\gamma,y),G_\delta(\gamma,y))^\zeta}  d \mu |_{\gamma _{0,i}}  \\&\leq& R(\delta)^\zeta \left\vert \int {1}  d \mu |_{\gamma _{0,i}} \right\vert \\&\leq& R(\delta)^\zeta||\mu |_{\gamma _{0,i}}||_W.
\end{eqnarray*}

\end{proof}
For the next, proposition and henceforth, for a given path $\Gamma _\mu ^\omega \in \Gamma_{ \mu }$ (associated with the disintegration $\omega = (\{\mu _\gamma\}_\gamma, \phi _1)$, of $\mu$), unless written otherwise, we consider the particular path $\Gamma_{\func{F_*}\mu} ^\omega \in \Gamma_{\func{F_*}\mu}$ defined by  Corollary \ref{oierew} by the expression

\begin{equation}
\Gamma_{\func{F_*}\mu} ^\omega (\gamma)=\sum_{i=1}^{\deg(f)}{\func{F}%
	_{\gamma _i \ast }\Gamma _\mu ^\omega (\gamma_i)\rho _i(\gamma _i)}\ \ m_{1}%
\mathnormal{-a.e.}\ \ \gamma \in M.  \label{niceformulaaareer}
\end{equation}Recall that $\Gamma_{\mu} ^\omega (\gamma) = \mu|_\gamma:= \pi_{2*}(\phi_{1}(\gamma)\mu _\gamma)$ and in particular $\Gamma_{\func{F_*}\mu} ^\omega (\gamma) = (\func{F_*}\mu)|_\gamma = \pi_{2*}(\func{P}_f\phi_1(\gamma)\mu _\gamma)$, where $\phi_1 = \dfrac{d \pi _{1*} \mu}{dm_1}$ and $\func{P}_f$ is the Perron-Frobenius operator of $f$.
\begin{proposition}\label{iuaswdas}
	If $F:\Sigma \longrightarrow \Sigma$ satisfies (f1), (f2), (f3), (G1), (G2) and $(\alpha \cdot L_1)^\zeta<1$, then there exist $0<\beta<1$ and $D >0$, such that for all $\mu \in \mathcal{H} _\zeta^{+} $ which satisfy $\phi _1 = 1$ $m_1$-a.e. and for all $\Gamma ^\omega _\mu \in \Gamma _\mu$, it holds $$|\Gamma_{\func{F_*} } ^\omega\mu|_{\zeta}  \leq \beta |\Gamma_{\mu}^\omega|_\zeta + D_2||\mu||_\infty,$$ for $\beta:= (\alpha L_1)^\zeta$ and $D_2:=\{\epsilon _\rho L_1^\zeta + |G|_ \zeta L_1^\zeta\}$.
\end{proposition}

\begin{corollary}\label{kjdfhkkhfdjfh}
	Suppose that $F:\Sigma \longrightarrow \Sigma$ satisfies (f1), (f2), (f3), (G1), (G2) and $(\alpha \cdot L_1)^\zeta<1$. Then, for all $\mu \in \mathcal{H} _\zeta^{+} $ which satisfy $\phi _1 = 1$ $m_1$-a.e. and $||\func{F_*}\mu||_\infty \leq ||\mu||_\infty$, it holds 
	\begin{equation}\label{erkjwr}
	|\Gamma_{\func{F_*}^n\mu}^\omega|_{\zeta}  \leq \beta^n |\Gamma _\mu^\omega|_\zeta + \dfrac{D_2}{1-\beta}||\mu||_\infty,
	\end{equation}
	for all $n\geq 1$, where $\beta$ and $D_2$ are from Proposition \ref{iuaswdas}.
\end{corollary}

\begin{remark}\label{kjedhkfjhksjdf}
	Taking the infimum over all paths $\Gamma_{ \mu } ^\omega  \in \Gamma_{ \mu }$ and all $\Gamma_{\func{F_*}^n\mu}^\omega  \in \Gamma_{\func{F_*}^n\mu}$ on both sides of inequality (\ref{erkjwr}), we get 
	
	\begin{equation}\label{fljghlfjdgkdg}
	|\func{F_*}^n\mu|_{\zeta}  \leq \beta^n |\mu|_\zeta + \dfrac{D_2}{1-\beta}||\mu||_\infty. 
	\end{equation}The above Equation (\ref{fljghlfjdgkdg}) will provide a uniform bound (see the proof of Theorem \ref{riirorpdf}) for the H\"older's constant of the measure $\func {F_*}^{n} m$, for all $n$ where $m$ is defined, as the product $m=m_1 \times \nu$, for a fixed probability measure $\nu$ on $K$. The uniform bound will be useful later on.
	
\end{remark}

\begin{remark}\label{riirorpdf}
	Consider the probability measure $m$ defined in Remark \ref{kjedhkfjhksjdf}, i.e., $m=m_1 \times \nu$, where $\nu$ is a given probability measure on $K$ and $m_1$ is the $f$-invariant measure fixed in Section \ref{intro}. Besides that, consider its trivial disintegration $\omega_0 =(\{m_{\gamma}  \}_{\gamma}, \phi_1)$, given by $m_\gamma = \func{\pi _{2,\gamma}^{-1}{_*}}\nu$, for all $\gamma$ and $\phi _1 \equiv 1$. According to this definition, it holds that 
	\begin{equation*}
	m|_\gamma = \nu, \ \ \forall \ \gamma.
	\end{equation*}In other words, the path $\Gamma ^{\omega _0}_m$ is constant: $\Gamma ^{\omega _0}_m (\gamma)= \nu$ for all $\gamma$.  Moreover, for each $n \in \mathbb{N}$, let $\omega_n$ be the particular disintegration of the measure $\func{F{_\ast }}^nm$ defined from $\omega_0$ as an application of Lemma \ref{transformula}, and consider the path $\Gamma^{\omega_{n}}_{\func{F{_\ast }}^n m}$ associated with this disintegration. By Proposition \ref{niceformulaab}, we have

	\begin{equation}
	\Gamma^{\omega_{n}}_{\func{F{_\ast }} ^n m} (\gamma)  =\sum_{i=1}^{s}{\dfrac{\func{F^n%
				_{f_{i}^{-n}(\gamma )}}_{\ast  }\nu}{|\det Df^n\circ f_{i}^{-n}(\gamma ))|}\chi _{f^n_i(P _{i})}(\gamma )}\ \ m_1-\hbox{a.e} \ \ \gamma \in M,  \label{niceformulaaw}
	\end{equation}where $P_i$, $i=1, \cdots, s=s(n)$, ranges over the partition $\mathcal{P}^{(n)}$ defined in the following way: for all $n \geq 1$, let $\mathcal{P}^{(n)}$ be the partition of $I$ s.t. $\mathcal{P}^{(n)}(x) = \mathcal{P}^{(n)}(y)$ if and only if $\mathcal{P}^{(1)}(f^j (x)) = \mathcal{P}^{(1)}(f^j(y))$ for all $j = 0, \cdots , n-1$, where $\mathcal{P}^{(1)} = \mathcal{P}$ (see remark \ref{chkjg}). This path will be used in the next section \ref{kjrthkje}.

The following result is an estimate for the regularity of the invariant measure of $F$ and its proof can be found in Theorem 7.5 of \cite{RRR}. This sort of result has many applications and can also be found in \cite{GLU} and \cite{LiLu}, wherein \cite{LiLu} the authors reach an analogous result for random dynamical systems.

\end{remark}

\begin{theorem}
	Suppose that $F:\Sigma \longrightarrow \Sigma$ satisfies (f1), (f2), (f3), (G1), (G2) and $(\alpha \cdot L_1)^\zeta<1$ and consider the unique $F$-invariant probability $\mu _{0}\in S^\infty$. Then $\mu _{0}\in \mathcal{H} _\zeta^{+}$ and 
	\begin{equation*}
	|\mu _{0}|_\zeta \leq \dfrac{D_2}{1-\beta},
	\end{equation*}where $D_2$ and $\beta$ are from Proposition \ref{iuaswdas}.
	\label{regg}
\end{theorem}

\section{Properties of \textbf{Admissible $R(\delta)$-Perturbations} }\label{kjrthkje}

In this section, we will prove some properties about \textbf{admissible $R(\delta)$-perturbations} perturbations. These properties will be used in the following sections, specifically to prove Theorem \ref{htyttigu}.

\begin{lemma}\label{UF2ass}
	Let $\{F_\delta \}_{\delta \in [0,1)}$ be an admissible $R(\delta)$-perturbation. Denote by $\func{F_\delta{_\ast}}$ their transfer operators, and by $%
	\mu_{\delta }$ their fixed points (probabilities) in $S^\infty$. Suppose that the family $\{\mu_{\delta }\}_{\delta \in [0,1)}$ satisfies 
	$$|\mu_{\delta }|_\zeta \leq B_u,$$for all $\delta \in [0,\delta_1)$.
	Then, there is a constant $C_{1}$ such that, it holds
	$$
	||(\func{F_0{_\ast }}-\func{F_\delta{_\ast }})\mu_{\delta }||_{{\infty}}\leq
	C_{1}R(\delta)^\zeta,$$for all $\delta \in [0,\delta_1)$, where $C_1:=|G_0|_\zeta + 3B_u +2$.
\end{lemma}

\begin{proof}

	Let us estimate
	
	\begin{equation}\label{12112}
		||(\func{F_0{_\ast }}-\func{F_\delta{_\ast }})\mu_{\delta }||_{\infty}= \esssup_{M}{||(\func{F_0{_\ast }}\mu_{\delta })|}_{\gamma }-(\func{F_\delta{_\ast }}\mu_{\delta })|_{\gamma}||_{W}.
	\end{equation}

	Denote by $f_{\delta,i}$, with $1\leq i\leq \deg(f)$, the branches of $f_{\delta}$ defined in the sets $P_{i} \in \mathcal{P}$(where  $\mathcal{P}$ depends on $\delta$), $f_{\delta,i}=f_{\delta }|_{P_{i}}$. Moreover,  remember that we denote $\gamma_{\delta, i}:= f^{-1}_{\delta, i} (\gamma)$ for all $\gamma \in M$, and by (U2.2) there exists $R(\delta)$ such that  
	
	\begin{equation}
		d_1(\gamma _{0,i},\gamma _{\delta,i}) \leq R(\delta) \ \ \forall i=1 \cdots \deg(f).
	\end{equation}We also recall that by (U1) $\deg (f_\delta) = \deg (f)$ for all $\delta \in [0,\delta_1)$.

	Thus, denoting $\func{F}_{\delta,\gamma_{\delta,i}}:=\func{F}_{\delta ,f_{\delta ,i}^{-1}(\gamma )}$ and $\mu := \mu_\delta$, we get

	\[
	({\func{F_0{_\ast }}}\mu -{\func{F_\delta{_\ast }}}\mu )|_{\gamma }=\sum_{i=1}^{\deg(f)}%
	\frac{{\func{F}_{0,\gamma _{0,i} }{_\ast }}\mu |_{\gamma _{0,i}}}{\det Df_{0}(\gamma _{0,i})}%
	-\sum_{i=1}^{\deg(f)}%
	\frac{{\func{F}_{\delta,\gamma _{\delta,i} }{_\ast }}\mu |_{\gamma _{\delta,i}}}{\det Df_{\delta}(\gamma _{\delta,i})},\ \ \mu _{x}-a.e.\ \gamma \in M. 
	\]%
	Then, we have
	
		\begin{equation*}
		||(\func{F_0{_\ast }}-\func{F_\delta{_\ast }})\mu||_{\infty} \leq \func{I}    +   \func{II},
	\end{equation*}where 
	
	\begin{equation}\label{I}
		\func{I} := \esssup_{M}\left\vert \left\vert \sum_{i=1}^{\deg(f)}%
		\frac{{\func{F}_{0,\gamma _{0,i} }{_\ast }}\mu |_{\gamma _{0,i}}}{\det Df_{0}(\gamma _{0,i})}%
		-\sum_{i=1}^{\deg(f)}%
		\frac{{\func{F}_{\delta,\gamma _{\delta,i} }{_\ast }}\mu |_{\gamma _{0,i}}}{\det Df_{\delta}(\gamma _{\delta,i})}\right\vert \right\vert _{W} 
	\end{equation}

	and 
	
	\begin{equation}	\label{II}
		\func{II} := \esssup_{M}\left\vert \left\vert \sum_{i=1}^{\deg(f)}%
		\frac{{\func{F}_{\delta,\gamma _{\delta,i} }{_\ast }}\mu |_{\gamma _{0,i}}}{\det Df_{\delta}(\gamma _{\delta,i})}%
		-\sum_{i=1}^{\deg(f)}%
		\frac{{\func{F}_{\delta,\gamma _{\delta,i} }{_\ast }}\mu |_{\gamma _{\delta,i}}}{\det Df_{\delta}(\gamma _{\delta,i})|}\right\vert \right\vert _{W}.
	\end{equation}
	
	Let us estimate $\func{I}$ of equation (\ref{I}). An analogous application of the triangular inequality, we have
	
	$$ \func{I} \leq \esssup_{M}\func{I}_a(\gamma) + \esssup_{M}\func{I}_b(\gamma),$$ where

	\begin{equation}
		\func{I}_a(\gamma) := \left\vert \left\vert \sum_{i=1}^{\deg(f)}%
		\frac{{\func{F}_{0,\gamma _{0,i} }{_\ast }}\mu |_{\gamma _{0,i}}}{\det Df_{0}(\gamma _{0,i})}
		-\sum_{i=1}^{\deg(f)}\frac{{\func{F}_{\delta,\gamma _{\delta,i} }{_\ast }}\mu |_{\gamma _{0,i}}}{\det Df_{0}(\gamma _{0,i})}\right\vert \right\vert _{W}
	\end{equation}and

	\begin{equation}
		\func{I}_b(\gamma) := \left\vert \left\vert \sum_{i=1}^{\deg(f)}%
		\frac{{\func{F}_{\delta,\gamma _{\delta,i} }{_\ast }}\mu |_{\gamma _{0,i}}}{\det Df_{0}(\gamma _{0,i})}
		-\sum_{i=1}^{\deg(f)}\frac{{\func{F}_{\delta,\gamma _{\delta,i} }{_\ast }}\mu |_{\gamma _{0,i}}}{\det Df_{\delta}(\gamma _{\delta,i})}\right\vert \right\vert _{W}.
	\end{equation}The summands will be treated separately. 
	
	For $\func{I}_a $, we note that

	\begin{eqnarray*}
		\func{I}_a(\gamma) &\leq &  \sum_{i=1}^{\deg(f)}%
		\left\vert \left\vert \frac{{\func{F}_{0,\gamma _{0,i} }{_\ast }}\mu |_{\gamma _{0,i}}}{\det Df_{0}(\gamma _{0,i})}
		-\sum_{i=1}^{\deg(f)}\frac{{\func{F}_{\delta,\gamma _{\delta,i} }{_\ast }}\mu |_{\gamma _{0,i}}}{\det Df_{0}(\gamma _{0,i})}\right\vert \right\vert _{W}
		\\&\leq &  \sum_{i=1}^{\deg(f)}%
		\frac{\left\vert \left\vert ({\func{F}_{0,\gamma _{0,i} }{_\ast }}- \func{F}_{\delta,\gamma _{\delta,i} }{_\ast })\mu |_{\gamma _{0,i}}\right\vert \right\vert _{W}}{\det Df_{0}(\gamma _{0,i})}
		\\&\leq &  \sum_{i=1}^{\deg(f)}%
		\frac{\left\vert \left\vert ({\func{F}_{0,\gamma _{0,i} }{_\ast }}- \func{F}_{0,\gamma _{\delta,i} }{_\ast })\mu |_{\gamma _{0,i}}\right\vert \right\vert _{W}}{\det Df_{0}(\gamma _{0,i})} +  \sum_{i=1}^{\deg(f)}%
		\frac{\left\vert \left\vert ({\func{F}_{0,\gamma _{\delta,i} }{_\ast }}- \func{F}_{\delta,\gamma _{\delta,i} }{_\ast })\mu |_{\gamma _{0,i}}\right\vert \right\vert _{W}}{\det Df_{0}(\gamma _{0,i})}.
	\end{eqnarray*}Now we note that $\mu$, satisfy $\phi_1 \equiv 1$. By Remark \ref{toyiout}, Corollary \ref{nslfdflsdjlf} and Lemma \ref{çhjghljk} applied on the last inequality above, we have
	
	\begin{eqnarray*}
		\func{I}_a(\gamma) &\leq &  \left(\sum_{i=1}^{\deg(f)}%
		\frac{1}{\det Df_{0}(\gamma _{0,i})}\right) R(\delta)^\zeta( 2 \alpha^\zeta |\mu|_\zeta + |G|_\zeta ||\mu ||_\infty ) \\&+& \left(\sum_{i=1}^{\deg(f)}%
		\frac{1}{\det Df_{0}(\gamma _{0,i})}\right)  R(\delta)^\zeta ||\mu |_{\gamma _{0,i}}||_W
		\\&\leq &    R(\delta)^\zeta (2B_u + |G_0|_\zeta +1). 
	\end{eqnarray*}
	
	For $\func{I}_b(\gamma)$, by (U2.1) we have
	\begin{eqnarray*}
		\func{I}_b(\gamma) &\leq&  \sum_{i=1}^{\deg(f)}%
		\left\vert \left\vert \frac{{\func{F}_{\delta,\gamma _{\delta,i} }{_\ast }}\mu |_{\gamma _{0,i}}}{\det Df_{0}(\gamma _{0,i})}
		-\frac{{\func{F}_{\delta,\gamma _{\delta,i} }{_\ast }}\mu |_{\gamma _{0,i}}}{\det Df_{\delta}(\gamma _{\delta,i})}\right\vert \right\vert _{W}
		\\&\leq& \sum_{i=1}^{\deg(f)}%
		\left\vert \frac{1}{\det Df_{0}(\gamma _{0,i})}
		-\frac{1}{\det Df_{\delta}(\gamma _{\delta,i})}\right\vert  \left\vert \left\vert {\func{F}_{\delta,\gamma _{\delta,i} }{_\ast }}\mu |_{\gamma _{0,i}}\right\vert \right \vert _{W}
		\\&\leq& \sum_{i=1}^{\deg(f)}%
		\left\vert \frac{1}{\det Df_{0}(\gamma _{0,i})}
		-\frac{1}{\det Df_{\delta}(\gamma _{\delta,i})}\right\vert
		\\&\leq&  R(\delta)^\zeta.
	\end{eqnarray*}Let us estimate $\func{II}$. By (Remark \ref{toyiout}), note that $\sum_{i=1}^{\deg(f)} \left\vert \frac{1}{\det Df_{\delta}(\gamma _{\delta,i})}\right\vert =1$ $m_1$-a.e.. Thus, we have
	\begin{eqnarray*}
		\func{II} &\leq&  \esssup_{M} \sum_{i=1}^{\deg(f)}%
		\left\vert \left\vert \frac{{\func{F}_{\delta,\gamma _{\delta,i} }{_\ast }}\mu |_{\gamma _{0,i}}}{\det Df_{\delta}(\gamma _{\delta,i})}
		-\frac{{\func{F}_{\delta,\gamma _{\delta,i} }{_\ast }}\mu |_{\gamma _{\delta,i}}}{\det Df_{\delta}(\gamma _{\delta,i})}\right\vert \right\vert _{W}
		\\&\leq&  \esssup_{M} \sum_{i=1}^{\deg(f)} \left\vert \frac{1}{\det Df_{\delta}(\gamma _{\delta,i})}\right\vert
		\left\vert \left\vert {\func{F}_{\delta,\gamma _{\delta,i} }{_\ast }}(\mu |_{\gamma _{0,i}}-\mu |_{\gamma _{\delta,i}})\right\vert \right\vert _{W}
		\\&\leq&  \esssup_{M} \sum_{i=1}^{\deg(f)} \left\vert \frac{1}{\det Df_{\delta}(\gamma _{\delta,i})}\right\vert
		\left\vert \left\vert \mu |_{\gamma _{0,i}}-\mu |_{\gamma _{\delta,i}}\right\vert \right\vert _{W}
		\\&\leq&  \esssup_{M} \sum_{i=1}^{\deg(f)} \left\vert \frac{1}{\det Df_{\delta}(\gamma _{\delta,i})}\right\vert
		d_1(\gamma _{\delta,i},\gamma _{0,i})^\zeta|\mu|_\zeta
		\\&\leq& \esssup_{M} \sum_{i=1}^{} \left\vert \frac{1}{\det Df_{\delta}(\gamma _{\delta,i})}\right\vert
		R(\delta) ^\zeta |\mu|_\zeta
		\\&\leq&  R(\delta) ^\zeta  B_u.
	\end{eqnarray*}Since $\zeta <1$, then $\delta \leq \delta ^\zeta$. Thus, all these facts yield
	\begin{eqnarray*}
		||(\func{F_0{_\ast }}-\func{F_\delta{_\ast }})\mu_{\delta }||_{\infty} & \leq & \func{I}  +   \func{II}
		\\& \leq & \func{I}_a    + \func{I}_b +   \func{II}
		\\& \leq & R(\delta)^\zeta (2B_u + |G_0|_\zeta +1) +  R(\delta)^\zeta + R(\delta) ^\zeta  B_u
		\\& \leq &C_1 R(\delta) ^\zeta,
	\end{eqnarray*}where $C_1:=|G_0|_\zeta + 3B_u +2.$
\end{proof}

The following result is an important tool to reach Theorem \ref{htyttigu}. It states that the function 
$$
\delta \longmapsto |\mu_{\delta }|_\zeta
$$ (see Definition \ref{Lips3}) is uniformly bounded, where $\{\mu_\delta\}_{\delta \in [0,1)}$ is the family of $F_{\delta }$-invariant probabilities of an admissible perturbation $\{F_{\delta }\}_{\delta \in [0,1)}$ of $F(=F_0)$.  

\begin{lemma}
	\label{thshgf}
	Let $\{F_{\delta }\}_{\delta \in [0,1)}$ be an admissible $R(\delta)$-perturbation and $\mu_{\delta }$ be the unique $F_\delta$-invariant probability in $S^\infty$, for all $\delta \in[0,1)$. Then, there exists $B_u>0$ such that 
	\begin{equation*}
		|\mu_{\delta }|_\zeta\leq B_u,
	\end{equation*}for all $\delta \in[0,1)$. 
\end{lemma}

First, we need a preliminary sublemma.

\begin{sublemma} If $\{F_\delta\}_{\delta \in [0,1)}$ is an admissible $R(\delta)$-perturbation. Then, there exist uniform constants $0<\beta_u<1$ and $D_{2,u}>0$ such that for every $\mu \in \mathcal{H} _\zeta^{+}$ which satisfies $\phi _1 = 1$ $m_1$-a.e., it holds
	\begin{equation}\label{er}
		|\Gamma_{\func{F_\delta {_*}}^n\mu}^\omega|_{\zeta}  \leq \beta_u ^n |\Gamma _\mu^\omega|_\zeta + \dfrac{D_{2,u}}{1-\beta_\delta}||\mu||_\infty,
	\end{equation}for all $\delta \in [0,1)$ and all $n \geq 0$.
	\label{las123rtryrdfd2}
\end{sublemma}

\begin{proof}
	We apply Corollary \ref{kjdfhkkhfdjfh} to each $F_\delta$ and obtain, 
	\begin{equation*}
		|\func{F_\delta{_\ast} \mu}|_\zeta \leq  \beta_\delta |\mu |_\zeta + D_{2,\delta} ||\mu||_{\infty}, \ \forall \delta \in [0,1),  \label{lasotaingt234dffggdgh2}
	\end{equation*}where $\beta_\delta:= (\alpha_\delta L_{1,\delta})^\zeta$ and $D_{2,\delta}:=\{\epsilon _{\rho,\delta}  L_{1,\delta}^\zeta + |G_\delta|_ \zeta L_{1,\delta}^\zeta\}$.

	By A2, we define $\beta_u:= \displaystyle{\sup _ \delta \beta_\delta}$ and $D_{2,u}:= \displaystyle{\sup_\delta D_{2,\delta}}$, and the result is established.
	
\end{proof}

\begin{proof}(of Lemma \ref{thshgf})

	Consider path $\Gamma^{\omega_n}_{\func{F_\delta{_\ast }^n}}m$, defined in Remark \ref{riirorpdf},  which represents the measure $\func{F_\delta {_\ast }}^nm$.

	According to Theorem \ref{belongss}, let $\mu_\delta \in S^\infty$ be the unique $F_\delta$-invariant probability measure in $S^\infty$. Consider the measure $m$, defined in Remark \ref{riirorpdf}, and its iterates $\func{F_{\delta \ast }}^n(m)$. By Theorem \ref{5.8}, these iterates converge to $\mu_\delta$ in $\mathcal{L}^{\infty}$. It implies that the sequence $\{\Gamma_{\func{F_{\delta \ast }}^n(m)} ^{\omega_n}\}_{n}$ converges $m_1$-a.e. to $\Gamma_{\mu_\delta}^\omega\in \Gamma_{\mu_\delta}$ (in $\mathcal{SB}(K)$ with respect to the metric defined in Definition \ref{wasserstein}), where $\Gamma_{\mu_\delta}^\omega$ is a path given by the Rokhlin Disintegration Theorem, and $\{\Gamma_{\func{F_\delta{\ast }}^n(m)} ^{\omega_n}\}_{n}$ is given by equation (\ref{niceformulaaw}). It implies that $\{\Gamma_{\func{F_\delta{\ast }}^n(m)} ^{\omega_n}\}_{n}$ converges pointwise to $\Gamma_{\mu_\delta}^\omega$ on a full measure set $\widehat{M_\delta}\subset M$.

	 Let us denote $%
	\Gamma_{n,\delta}:=\Gamma^{\omega_n}_{\func{F_\delta{_\ast }}^n(m)}|_{%
		\widehat{M_\delta}}$ and $\Gamma_\delta:=\Gamma^\omega _{\mu _{\delta}}|_{\widehat{M_\delta}}$. Since $\{\Gamma_{n,\delta} \}_n $ converges pointwise to $\Gamma _\delta$, it holds $|\Gamma_{n,\delta}|_\zeta \longrightarrow |\Gamma _\delta|_\zeta$ as $n \rightarrow \infty$. Indeed, let $x,y \in \widehat{M_\delta}$. Then,
	
	\begin{eqnarray*}
		\lim _{n \longrightarrow \infty} {\dfrac{||\Gamma_{n,\delta} (x) - \Gamma _{n, \delta}(y)||_W}{d_1(x,y)^\zeta}} &= & \dfrac{||\Gamma _\delta (x) - \Gamma _\delta (y)||_W}{d_1(x,y)^\zeta}.
	\end{eqnarray*} On the other hand, by Lemma \ref{thshgf}, the argument of the left-hand side is bounded by $|\Gamma_{n, \delta}|_\zeta \leq  \dfrac{D_u}{1-\beta_u}$  for all $n\geq 1$. Then, 
	\begin{eqnarray*}
		\dfrac{||\Gamma _\delta (x) - \Gamma _\delta (y)||_W}{d_1(x,y)^\zeta}&\leq & \dfrac{D_{u}}{1-\beta _u}.
	\end{eqnarray*} 
	Thus $|\Gamma^\omega_{\mu _{\delta}}|_\zeta \leq\dfrac{D_{u}}{1-\beta _u}$, and taking the infimum we get $|\mu _{\delta}|_\zeta \leq \dfrac{D_{u}}{1-\beta _u}$. We finish the proof defining $B_u:=\dfrac{D_{u}}{1-\beta _u}$.
\end{proof}

\section{Perturbation of Operators}\label{perturbationoperators}

The main results of this article (Theorems \ref{htyttigu} and Corollary \ref{htyttigui}) are proven by demonstrating that an \textbf{admissible} \textbf{$R(\delta)$-perturbation} induces a family of transfer operators, ${\func{T}_{\delta}}_{\delta \in [0,1)}$, referred to as the \textbf{$(R(\delta), \zeta)$-family of operators}, which is defined in the following paragraph. The main tool used to establish this is Lemma \ref{dlogd}, which is stated and proved in this section.

\begin{definition}
Suppose there are vector spaces $(B_{w}, ||\cdot||_{w} )$ and  $(B_{s}, ||\cdot||_{s} )$, satisfying $B_{s} \subset B_{w}$ and $||\cdot||_{s}\geq ||\cdot||_{w}$, where the actions $\func{T}  _{\delta }: B_{w} \longrightarrow B_{w}$, $\func{T}  _{\delta }: B_{s} \longrightarrow B_{s}$ are well defined and, for each $\delta \in [0,1)$, $\mu_{\delta }\in B_{s}$ is a fixed point for $\func{T}  _{\delta }$. Moreover, suppose that: 
\begin{enumerate}
	\item [(P1)] There are $C\in \mathbb{R}^+$ and a real-valued function $\delta \longmapsto R(\delta) \in \mathbb{R}^+$, defined on $[0,1)$,  such that $$\lim_{\delta \rightarrow 0^+} {R(\delta)\log (\delta)}=0$$ and%
	\begin{equation*}
		||(\func{T}  _{0}-\func{T}  _{\delta })\mu_{\delta }||_{w}\leq R(\delta) ^\zeta C \ \forall \delta \in [0,1); 
	\end{equation*}
	
	\item [(P2)] Suppose there is $M>0$ such that for all $\delta \in [0,1)$, it holds $$||\mu_{\delta}||_{s}\leq M;$$
	\item[(P3)] $\func{T} _{0}$ has exponential convergence to equilibrium with
	respect to the norms $||\cdot||_{s}$ and $||\cdot||_{w}$: there exists $0<\rho_2 <1$ and $C_2>0$ such that  $$\forall \ \mu \in \mathcal{V}_s: =\{\mu \in B_{s}: \mu(\Sigma)=0 \}$$ it holds $$||\func{T}^{n}_0 \mu||_{w} \leq \rho _2 ^n C_2 ||\mu||_{s};$$
	\item[(P4)] The iterates of the operators are uniformly bounded for the weak norm: there exists $M_2 >0$ such that for all $\delta \in [0,1)$,  all $n \in \mathbb{N}$, and all $\nu \in B_{s}$, it holds $||\func{T} _{\delta
	} ^{ n}\nu||_{w}\leq M_{2}||\nu||_{w}.$
\end{enumerate}A family of operators that satisfies (P1), (P2), (P3) and (P4) is called a \textbf{$(R(\delta), \zeta)$-family of operators}.
\end{definition}

The following Lemma \ref{dlogd} establishes a general and quantitative relation between the variation of the fixed points, $\{\mu_ \delta\}_{\delta \in [0,1)}$, of a $(R(\delta), \zeta)$-family of operators concerning the parameter $\delta$. It states that the function $\delta \mapsto \mu_{\delta }$, given by $$\delta \longmapsto \func{T}_{\delta } \longmapsto \mu_{\delta }, \ \ \delta \in [0,1)$$varies continuously at $0$, with respect to the norm $||\cdot||_w$, and provides an explicit bound for its modulus of continuity: $D_1 R(\delta)^\zeta \log \delta $, where $D_1\geq 0$.

\begin{lemma} [Quantitative stability for fixed points of operators]	\label{dlogd}
	Suppose $\{\func{T}_{\delta }\}_{\delta \in \left[0, 1 \right)}$ is a $(R(\delta),\zeta)$-family of operators, where $\mu_{0}$ is the unique fixed point of $\func{T}_{0}$ in $B_{w}$ and $%
	\mu_{\delta }$ is a fixed point of $\func{T} _{\delta }$. Then, there exist constants $D_1 < 0$ and $\delta _0 \in (0,1)$ such that for all $\delta \in [0,\delta _0)$, it holds
	
	\begin{equation*}
		||\mu_{\delta }-\mu_{0}||_{w}\leq D_1 R(\delta)^\zeta \log \delta.
	\end{equation*}
\end{lemma}

To prove Lemma \ref{dlogd}, we state a general result on the stability of fixed points. We will omit its proof, but the reader can find it for instance in \cite{GLU}, Lemma 12.1. 

Consider two operators $\func{T} _{0}$ and $\func{T}  _{\delta }$ preserving a normed space of signed measures $\mathcal{B\subseteq }\mathcal{SB}(X)$ with norm $||\cdot ||_{\mathcal{B}%
}$. Suppose that $f_{0},$ $f_{\delta }\in \mathcal{B}$ are fixed points of $\func{T}  _{0}$ and $\func{T}  _{\delta }$, respectively.

\begin{sublemma}
	\label{gen}Suppose that:
	
	\begin{enumerate}
		\item[a)] $||\func{T}  _{\delta }f_{\delta }-\func{T} _{0}f_{\delta }||_{\mathcal{B}}<\infty $;
		
		\item[b)] For all $i\geq 1$, $\func{T} _{0}^{ i}$ is continuous on $\mathcal{B}$: for each $ i \geq 1$, $\exists
		\,C_{i}~s.t.~\forall g\in \mathcal{B},~||\func{T}  _{0}^{ i}g||_{\mathcal{B}}\leq
		C_{i}||g||_{\mathcal{B}}.$
	\end{enumerate}
	
	Then, for each $N \geq 1$, it holds%
	\begin{equation*}
		||f_{\delta }-f_{0}||_{\mathcal{B}}\leq ||\func{T} _{0}^{ N}(f_{\delta }-f_{0})||_{%
			\mathcal{B}}+||\func{T}  _{\delta } f_{\delta }-\func{T}  _{0}f_{\delta }||_{\mathcal{B}%
		}\sum_{i\in \lbrack 0,N-1]}C_{i}.  
	\end{equation*}
\end{sublemma}

\begin{proof} (of Lemma \ref{dlogd})

	First, note that if $\delta \geq 0$ is small enough, then $\delta \leq -\delta \log {\delta} $. Moreover, $x -1 \leq \lfloor  x \rfloor$, for all $x \in \mathbb{R}$.

	By P1,%
	\begin{equation*}
		||\func{T}_{\delta }\mu_{\delta }-\func{T}_{0}\mu_{\delta }||_{w}\leq R(\delta) ^\zeta C
	\end{equation*}%
	(see Lemma \ref{gen}, item a) ) and P4 yields $C_{i}\leq M_{2}.$ 
	
	Hence, by Lemma \ref{gen} we have 
	\begin{equation*}
		||\mu_{\delta }-\mu_{0 }||_{w}\leq R(\delta) ^\zeta CM_{2}N+||\func{T}_{0}^{ N}(\mu_{0}-\mu_{\delta})||_{w}.
	\end{equation*}%
	By the exponential convergence to equilibrium of $\func{T}_{0}$ (P3), there exists $0<\rho _2 <1$ and $C_2 >0$ such that (recalling that
	by P2 $||(\mu_{\delta}-\mu_{0})||_{s}\leq 2M$)
	\begin{eqnarray*}
		||\func{T} _{0}^{ N}(\mu_{\delta }-\mu_{0})||_{w} &\leq &C_{2}\rho _{2}^{N}||(\mu_{\delta }-\mu_{0 })||_{s} \\
		&\leq &2C_{2}\rho _{2}^{N}M
	\end{eqnarray*}%
	hence%
	\begin{equation*}
		||\mu_{\delta }-\mu_{0}||_{w}\leq R(\delta) ^\zeta CM_{2}N+2C_{2}\rho _{2}^{N}M.
	\end{equation*}%
	Choosing $N=\left\lfloor \frac{\log \delta }{\log \rho _{2}}\right\rfloor $, we have%
	\begin{eqnarray*}
		||\mu_{\delta }-\mu_{0} ||_{w} &\leq &R(\delta) ^\zeta CM_{2}\left\lfloor \frac{%
			\log \delta }{\log \rho _{2}}\right\rfloor +2C_{2}\rho _{2}^{\left\lfloor 
			\frac{\log \delta }{\log \rho _{2}}\right\rfloor }M \\
		&\leq &R(\delta) ^\zeta \log \delta CM_{2}\frac{1}{\log \rho _{2}}+2C_{2}\rho _2 ^ {  \frac{\log \delta }{\log \rho _{2}} -1} M \\
		&\leq & R(\delta) ^\zeta \log \delta CM_{2}\frac{1}{\log \rho _{2}}+\frac{2C_{2}\rho _2 ^ { \frac{\log \delta }{\log \rho _{2}}} M}{\rho _2} \\
		&\leq & R(\delta) ^\zeta \log \delta CM_{2}\frac{1}{\log \rho _{2}}+\frac{2C_{2}\delta M}{\rho _2} \\
		&\leq & R(\delta) ^\zeta \log \delta CM_{2}\frac{1}{\log \rho _{2}}-\frac{2C_{2}\delta \log \delta M}{\rho _2}\\
		&\leq & R(\delta) ^\zeta \log \delta \left( \frac{CM_{2}}{\log \rho _{2}}-\frac{2C_{2} M}{\rho _2} \right).
		\notag
	\end{eqnarray*}
	We finish the proof by setting, $D_1= \frac{CM_{2}}{\log \rho _{2}}-\frac{2C_{2} M}{\rho _2}$.
\end{proof}

\section{Proof of Theorem \ref{belongss}}\label{sofkjsdkgfhksjfgd}

First, let us prove the existence and uniqueness of an $F$-invariant measure in $S^\infty$.

The following lemma \ref{kjdhkskjfkjskdjf} ensures the existence and uniqueness of an $F$-invariant measure that projects onto $m_1$. Since its proof is based on standard arguments (see \cite{AP}, for instance), we will omit it here.

\begin{lemma}\label{kjdhkskjfkjskdjf}
	There exists a unique measure $\mu_0$ on $M \times K$ such that for every continuous function $\psi \in C^0 (M \times K)$, it holds

	\begin{equation}
	\psi^+= \psi^-=\int {\psi}d\mu_0, 
	\end{equation}where $$\psi^- := \lim _{n \rightarrow \infty} {\int{\inf_{(\gamma, y) \in \gamma \times K} \psi \circ F^n (\gamma, y) }dm_1(\gamma)}$$ and  $$\psi^+ := \lim _{n \rightarrow \infty} {\int{\sup_{(\gamma, y) \in \gamma \times K} \psi \circ F^n (\gamma, y)}dm_1 (\gamma)}.$$  Moreover, the measure $\mu_0$ is $F$-invariant and $\pi_1{_\ast}\mu_0 = m_1$.
\end{lemma}

Let $\mu_{0}$ be the $F$-invariant measure such that $\pi_{1\ast }\mu_{0}=m_1$ (which exists by Lemma \ref{kjdhkskjfkjskdjf}), where $1$ is the unique $f$-invariant density in $H_\zeta$. Suppose that $g:K\longrightarrow \mathbb{R}$ is a $\zeta$-H\"older function such that $|g|_{\infty }\leq 1$ and $H_\zeta(g)\leq 1$. Then, it holds $\left\vert \int {g} d(\mu_{0}|_{\gamma })\right\vert \leq |g|_{\infty }\leq 1$. Hence, $\mu_{0}\in \mathcal{L}^{\infty }$. Since $\dfrac{\pi_{1*}\mu_0}{dm_1} \equiv 1$, we have $\mu_0 \in S^\infty$.

The uniqueness follows directly from Theorem \ref{5.8} since the difference between two probabilities ($\mu _1 - \mu_0$) is a zero-average signed measure.

\begin{definition}
	Let $F:\Sigma \longrightarrow \Sigma$ be a continuous map,  with $\Sigma=M\times K$ and $F(x,y)=(f(x), G(x,y))$, where $f:M\longrightarrow M$ and $G(x, \cdot ): K\longrightarrow K$ for all $x\in M$. We say that $E\subset \Sigma$ is an $(n,\varepsilon)-$spanning set if for every $(x_0, y_0)\in \Sigma$, there exists $(x_1,y_1)\in E$ such that, for all $j\in \{0,1,...,n-1\}$
	
	\begin{align*}
	d(F^{j}(x_0,y_0),F^{j}(x_1, y_1))&= d((f^{j}(x_0), G_{x_0}^{j}(y_0)),(f^{j}(x_1), G_{x_1}^{j}(y_1))\\
	& =d_1(f^{j}(x_0), f^{j}(x_1)) + d_2(G_{x_0}^{j}(y_0), G_{x_1}^{j}(y_1))\\
	& < \varepsilon,
	\end{align*}where $d_1$ and $d_2$ are the metrics on $M$ and $K$, respectively. For $\varphi \in C^{0}(M\times K, \mathbb{R})$(space of continuous functions), define the {\bf topological pressure} of $\varphi$ by
	$$
	P_{t}(F,\varphi) := \lim_{\varepsilon \to 0} \limsup_{n \to \infty}\dfrac{1}{n} \log \inf_{E \subset \Sigma} \biggl( \sum_{(x,y)\in E}\e ^{S_{n}\varphi(x,y)}\biggr)
	$$
	where $S_{n}(\varphi)(x,y):=\sum_{j=0}^{n-1} \varphi(F^{j}(x,y))=\sum^{n-1}_{j=0}\varphi(f^{j}(x), G^{j}_{x}(y))$, and the infinium is taken over all $(n,\varepsilon)-$ spanning subsets $E$ of $\Sigma$. 
\end{definition}
It is known that the variational principle holds.  That is,
\begin{equation}
\label{variaprinciple}
P_{t}(F,\varphi) = \sup_{\mu \in \mathcal{M}^{1}_{F}(M\times K)} \biggl\{h_{\mu}(F) + \int \varphi d\mu \biggr \}
\end{equation}
where $\mathcal{M}^{1}_{F}(M\times K)$ is the set of measures $\mu$ that are invariant by $F$ ($\mu \circ F^{-1}=\mu$). 
On the other hand, for a given $\varphi^{\ast}\in C^{0}(M,\mathbb{R})$ define the function
\[
\begin{array}{cccc}
\varphi\ : & \! M\times K & \! \longrightarrow
& \! \mathbb{R} \\
& \! (x,y) & \! \longmapsto
& \! \varphi(x,y):=\varphi^{\ast}(x).
\end{array}
\]
We have that $\varphi\in C^{0}(M\times K, \mathbb{R})$. Now, let $\mathcal{M}^{1}_{m_1}(M\times K)$ be the set of all probability measures $\mu$ on $M\times K$ such that
\[
\pi_{1\ast}\mu = \mu\circ \pi^{-1}_{1}=m_1. 
\]
Where $\pi_1:M \times K \rightarrow M$ stands for the first projection ($\pi_1(x,y)=x$). Theorem \ref{rok} (Rokhli's disintegration theorem) describes a disintegration $\left( \{\mu _{\gamma}\}_{\gamma }, m_1\right) $ of $\mu$. So that

\begin{align*}
\int_{M\times K} \varphi d\mu & = \int_M \int_K \varphi(\gamma,y)d\mu_\gamma(y) dm_1(\gamma)\\
& = \int_M \int_K \varphi^\ast(\gamma) d\mu_\gamma(y)dm_1(\gamma)\\
& = \int_M \varphi^\ast(\gamma) dm_1(\gamma) < \infty.
\end{align*}
If we consider an $(n,\varepsilon)$-spanning set $E\subset M\times K$, then by the metric $d$, $E^\ast =\{x\in M: (x,y)\in E\}$ is an $(n,\varepsilon)$- spanning set for the system $f:M\longrightarrow M$. Hence, by definition of topological pressure, we get
\begin{equation}
\label{equaPfPF1}
P_t(f, \varphi^{\ast})\leq P_{t}(F, \varphi).
\end{equation}

For the other inequality, we will use the following result (see \cite{PL}).
\begin{theorem}[Ledrappier-Walters Formula]
	\label{thLed-Walters}
	Let $\Hat{X}, X$ be compact metric spaces and let $\Hat{T}:\Hat{X}\longrightarrow \Hat{X}$, $T:X\longrightarrow X$ and $\Hat{\pi}:\Hat{X}\longrightarrow X$ be continuous maps such that $\Hat{\pi}$ is surjective and $\Hat{\pi}\circ \Hat{T} = T\circ \Hat{\pi}$. Then
	$$
	\sup_{\Hat{\nu}; \Hat{\pi}_{\ast}\Hat{\nu}=\nu}h_{\Hat{\nu}}(\Hat{T})= h_{\nu}(T) + \int h_{top}(\Hat{T}, \Hat{\pi}^{-1}(y)) d\nu(y).
	$$
\end{theorem}

Since $G(x, \cdot):K\longrightarrow K$ is a uniform contraction, for every $x\in M$, we have  $h_{top}(F, \pi_1^{-1}(x))=0$ for every $x\in M$. Then, by Theorem \ref{thLed-Walters}, we obtain
\begin{equation}
\label{equahfHF}
h_{\mu}(F)=h_{m_1}(f)
\end{equation}
for every $m_1\in \mathcal{M}_{f}(M)$ and $\mu\in \mathcal{M}_{F}(M\times K)$ such that $\pi_{1\ast}\mu=m_1$. Therefore, by (\ref{variaprinciple}) and (\ref{equahfHF}) we get

\begin{equation}
\label{equaPFpf2}
P_{t}(F,\varphi)\leq P_{t}(f,\varphi^\ast).
\end{equation}Combining (\ref{equaPfPF1}) and (\ref{equaPFpf2}) we get
\begin{equation}
\label{equaPFpf3}
P_{t}(F,\varphi)= P_{t}(f,\varphi^\ast).
\end{equation}

\begin{proposition}
	\label{propestadoequil}
	The measure $m_1\in \mathcal{M}_{f}(M)$ is an equilibrium state for $(f, \varphi^\ast)$, if and only if, $\mu\in \mathcal{M}_{F}(M\times K)$ such that $m_1=\pi_{1\ast} \mu$, is an equilibrium state for $(F,\varphi)$. Moreover, if $m_1$ is the unique equilibrium state, then $\mu$ is unique.

\end{proposition}

\begin{proof} (of Theorem \ref{belongss})
The proof of the theorem follows from (\ref{equahfHF}) and (\ref{equaPFpf3}). For the second part, it is a consequence of Lemma \ref{kjdhkskjfkjskdjf}.
\end{proof}

\section{Proof of Theorem  \ref{htyttigu} and Corollary \ref{htyttigui}}\label{kkdjfkshfdsdfsttr}
Before to establish Theorem \ref{htyttigu}, we need to prove the following Lemma \ref{rrr}.

\begin{lemma}\label{rrr}
	Let $\{F_\delta\}_{\delta \in [0,1)}$ be an admissible $R(\delta)$-perturbation and let $\{\func{F_\delta{_*}}\}_{\delta \in [0,1)}$ be the induced family of transfer operators. Then, $\{\func {F_\delta{_*}}\}_{\delta \in [0,1)}$ is an $(R(\delta),\zeta)$-family of operators with weak space $(\mathcal{L}^{\infty}, || \cdot ||_\infty)$ and strong space $(S^\infty,||\cdot||_{S^\infty})$.
\end{lemma}

\begin{proof}
		We need to prove that $\{F_\delta\}_{\delta \in [0,1)}$ satisfies P1, P2, P3 and P4. To prove P2, note that, by (A1) and equation (\ref{weakcontral11234}) we have 
	\begin{eqnarray*}
		||\func{F {_\delta} {_*} ^n}  \mu_{\delta }||_{S^\infty} &=& |\func {P}_{f_{\delta }} ^n \phi _1 |_{\zeta} + ||\func{F_\delta{_\ast }} ^n \mu||_{\infty}  \\&\leq & D \lambda ^n |\phi _1|_{\zeta} + D|\phi _1|_\infty + || \mu||_{\infty} \\&\leq & D \lambda ^n ||\mu||_{S^\infty} + (D +1) || \mu||_{\infty}.
	\end{eqnarray*}
	Therefore, if $\mu_{\delta }$ is a fixed probability measure for the operator $\func{F_\delta {_*}}$, by the above inequality, we get P2 with $M=D+1$.
	
	A direct application of Theorem \ref{thshgf} and Lemma \ref{UF2ass} gives P1. The items P3 and P4 follow, respectively, from proposition \ref{5.8}, equation (\ref{weakcontral11234}) applied to each $F_\delta$.
\end{proof}	
\begin{proof}(of Theorem \ref{htyttigu} and Corollary \ref{htyttigui})

We directly apply the above results together with Theorem \ref{dlogd}, and the proof of Theorem \ref{htyttigu} is completed. The proof of Corollary \ref{htyttigui} is straightforward.
\end{proof}

\end{document}